\documentclass[reqno]{amsart}
\usepackage[utf8]{inputenc}
\usepackage{amscd,amsfonts,amsmath,amssymb,amsthm,bbm,dsfont,centernot,mathtools}
\usepackage{color,mathrsfs,stackrel}
\usepackage[margin=1in]{geometry}
\usepackage[colorlinks=true,linkcolor=blue,citecolor=red]{hyperref}
\usepackage{yhmath}
\usepackage{fancyhdr}
\usepackage{xy}
\usepackage{eso-pic}
\usepackage{xcolor}
\usepackage{extarrows}
\usepackage{relsize}
\usepackage{caption}
\usepackage{eurosym}
\usepackage{chngcntr}
\usepackage{tikz}

\newtheorem{corollary}{Corollary}[section]
\newtheorem{lemma}[corollary]{Lemma}
\newtheorem{proposition}[corollary]{Proposition}
\newtheorem{theorem}[corollary]{Theorem}

\theoremstyle{definition}
\newtheorem{definition}[corollary]{Definition}
\newtheorem{remark}[corollary]{Remark}

\newtheorem*{acknowledgements}{\sc Acknowledgements}

\numberwithin{equation}{section}

\allowdisplaybreaks

\newcommand{\norm}[1]{\lVert #1 \rVert}
\newcommand{\spr}[2]{\langle #1, #2\rangle}

\def\Xint#1{\mathchoice
{\XXint\displaystyle\textstyle{#1}}%
{\XXint\textstyle\scriptstyle{#1}}%
{\XXint\scriptstyle\scriptscriptstyle{#1}}%
{\XXint\scriptscriptstyle\scriptscriptstyle{#1}}%
\!\int}
\def\XXint#1#2#3{{\setbox0=\hbox{$#1{#2#3}{\int}$ }
\vcenter{\hbox{$#2#3$ }}\kern-.6\wd0}}
\def\dashint{\Xint-}
\def\mint{\Xint{\rotatebox[origin][30]{$-$}}}

\DeclareMathAlphabet{\mdutchcal}{U}{dutchcal}{m}{n}

\def \div {\mathop {\rm div}\nolimits}

\def \de {\mathrm{d}}
\def \e {\mathrm{e}}
\def \R {\mathbb R}
\def \B {\mathbb B}
\def \O {\Omega }
\def \D {\mathbb D}
\def \A {\mathbb A}
\def \N {\mathbb N}

\def \E {\mathcal{E}}
\def \OG {\Omega\setminus\Gamma}
\def \V {\mathcal{V}}

\def \D {\mathcal{D}^D}
\def \En {\mathscr{E}}
\def \Dis {\mathscr{D}}
\def \W {\mathscr{W}}
\def \Li {\mathcal{L}\mdutchcal{ip}^D(0,T)}
\def \L {\mathcal{L}\mdutchcal{ip}^D_{0,T}(0,T)}
\def \LT {\mathcal{L}\mdutchcal{ip}^D_{T}(0,T)}
\def \tul {\tilde{u}}
\def \dtul {\dot{\tilde{u}}}
\def \bul {\bar{u}}
\def \dbul {\dot{\bar{u}}}

\def \vs {\textbf{v}}

\begin{document}
\title{A dynamic model for viscoelasticity in domains with time-dependent cracks}
\author[F. Sapio]{Francesco Sapio}
 \address[Francesco Sapio]{SISSA, via Bonomea 265, 34136 Trieste, Italy}
 \email{fsapio@sissa.it}

\begin{abstract}
In this paper, we prove the existence of solutions for a class of viscoelastic dynamic systems on time-dependent cracking domains.
\end{abstract}

\maketitle


\medskip

	
\section{Introduction}
In this paper we study the dynamic crack growth in viscoelastic materials with long memory. When no crack is present, important contributions in the theory of linear viscoelasticity are due to such scientists as Maxwell, Kelvin, and Voigt. Their names are associated with two well-known models of dissipative solids which can be described in terms of a spring and a dash-pot in series (Maxwell's model) or in parallel (Kelvin-Voigt's model), see \cite{SL}. Boltzmann was the first to develop a three-dimensional theory of isotropic viscoelasticity in \cite{Bo}, and later Volterra in \cite{Vo} obtained similar results for anisotropic solids. 

In literature we can find two different classes of materials in the case of viscoelastic deformations: materials with short memory and materials with long memory. The first case is associated to a local model, which means that the state of stress at the instant $t$ only depends on the strain at that instant. In the second case, instead, the associated model is non-local in time, in the sense that the state of stress at the instant $t$ depends also on the past history up to time $t$ of the strain. According to \cite{Fab, MF-AM}, in the case of viscoelastic materials with long memory the general stress-strain relation is the following
\begin{equation*}
    \sigma(t,x):=G(0,x)\nabla u(t,x)+\int_{-\infty}^t G'(t-\tau,x)\nabla u(\tau,x)\de\tau,\qquad t\in(-\infty,T],\hspace{0.1cm}x\in \Omega,
\end{equation*}
for a suitable choice of the memory kernel $G$, and with some prescribed boundary conditions. 

To describe our model we start with a short description of the standard approach to dynamic fracture in the case of linearly elastic materials with no viscosity. In this situation, the deformation of the elastic part of the material evolves according to elastodynamics; for an antiplane displacement, elastodynamics together with the stress-strain relation $\sigma(t,x)=\nabla u(t,x)$, leads to the following wave equation
\begin{equation}\label{wave_eq}
    \ddot u(t,x)-\div \sigma(t,x)=f(t,x),\qquad t\in[0,T],\hspace{0.1cm}x\in\OG_t,
\end{equation}
 with some prescribed boundary and initial conditions. Here, $\Omega\subset\R^d$ is a bounded open set, which represents the cross-section of the body in the reference configuration, $\Gamma_t\subset\overline\Omega$ models the cross-section of the crack at time $t$, $u(t,\cdot)\colon \OG_t\to\R$ is the antiplane displacement, and $f(t,\cdot):\OG_t\rightarrow\R$ is a forcing term. From the mathematical point of view, a first step towards the study of the evolution of fractures is to solve the wave equation~\eqref{wave_eq} when the time evolution of the crack is assigned, see for example~\cite{C2,DM-Lar,DM-Luc,NS}.

In this paper, we consider Maxwell's model in the case of dynamic fracture, when the crack evolution $t\mapsto\Gamma_t$ is prescribed. In this case, the memory kernel $G$ has an exponential form (see for example \cite{SL}), and the displacement satisfies the following equation
\begin{equation}\label{viscoelastic*}
   \ddot u(t,x)-(c_1+c_2)\Delta u(t,x)+c_2\int_{-\infty}^t\e^{-(t-\tau)}\Delta u(\tau,x)\de \tau=f(t,x),\qquad t\in(-\infty,T],\hspace{0.1cm}x\in\OG_t,
\end{equation}
where $c_1$ and $c_2$ are two positive constants. As in \cite{DA2,Fab}, we suppose that the past history of the displacement up to time $0$ is already known, therefore, it is convenient to write equation \eqref{viscoelastic*} as 
\begin{align}\label{viscoelasticM}
   \ddot u(t,x)-(c_1+c_2)\Delta u(t,x)&+c_2\int_{0}^t\e^{-(t-\tau)}\Delta u(\tau,x)\de \tau\nonumber\\
   &=f(t,x)-c_2\int_{-\infty}^0\e^{-(t-\tau)}\Delta v(\tau,x)\de \tau,\qquad t\in[0,T],\hspace{0.1cm}x\in\OG_t,
\end{align}
where the function $v$ represents the past history, that is $v(t,x)=u(t,x)$ for every $t\in(-\infty,0]$ and $x\in\OG_t$.

The main results of this paper are Theorem~\ref{main1} and Theorem~\ref{main2}, in which we prove, by two different methods, the existence of a solution to~\eqref{viscoelasticM}. This is done not only in the antiplane case, but also in the more general case of linear elasticity in dimension $d$; that is, when the displacement is vector-valued and the elastic energy depends on the symmetrized gradient of the displacement. 

The first method, considered in Theorem~\ref{main1}, is based on a generalization of Lax-Milgram's Theorem (\cite[Chapter 3, Theorem 1.1]{LL}). We follow the lines of the proof of Theorem 2.1 in~\cite{DA}. In doing so, the main difficulty is given by the fact that the set $\OG_t$, where equation \eqref{viscoelasticM} holds, depends on time. This requires the introduction of suitable function spaces used to adapt the proof in~\cite{DA}.  

 The second method, provided by Theorem~\ref{main2}, is based on a time discretization scheme that yields a solution which, in addition, satisfies the energy-dissipation inequality~\eqref{eq:eninM}. This procedure, adopted in~\cite{DM-Lar} for wave equation \eqref{wave_eq} in a time-dependent domain, consists of the following steps: time discretization, construction of an approximate solution, discrete energy estimates, and passage to the limit.
 
The main difficulty in applying this procedure, in the same way it was done in~\cite{DM-Lar}, is the identification of the term in the energy-dissipation inequality which corresponds to the non–local in time viscous term $\int_0^t \e^{-(t-\tau)}\Delta u(\tau,x)\de \tau$ appearing in \eqref{viscoelasticM}.
  
To fix this issue, we introduce an auxiliary variable $w$ and we transform our equation \eqref{viscoelasticM} into an equivalent system (see Definition \ref{defsym}) of two equations in the two variables $u$ and $w$, without long memory terms, which has to be solved on the time-dependent domain $\OG_t$. The advantage of this strategy lies in the fact that we transform a non-local model (the equation) into a local one (the system). 

We discretize the time interval $[0,T]$ by using the time step  $\tau_n:=\frac{T}{n}$. To define the approximate solution $(u_n,w_n)$ at time $(k+1)\tau_n$, we solve an incremental problem (see \eqref{unkM}) depending on the values of $(u_n,w_n)$ at times $(k-1)\tau_n$ and $k\tau_n$. Since the new system has a natural notion of energy, we also obtain a discrete energy estimate for $(u_n,w_n)$. Then, we extend $(u_n,w_n)$ to the whole interval $[0,T]$ by a suitable interpolation, and by using the energy estimates together with a compactness result we pass to the limit, along a subsequence of $(u_n,w_n)$. It is now possible to prove that the limit of this subsequence of $(u_n,w_n)$ is a solution to the system, which is equivalent to our equation \eqref{viscoelasticM}. As a byproduct, from the discrete energy estimates we obtain the energy-dissipation inequality~\eqref{eq:eninM}.

The paper is organized as follows. In Section~\ref{visc_sec2} we fix the notation adopted throughout the paper. In Section~\ref{visc_sec3M} we list the standard assumptions on the family of cracks $\{\Gamma_t\}_{t\in[0,T]}$, we state the evolution problem in the general case, and we specify the notion of solution to the problem. In Section~\ref{sub1} and~\ref{sub2} we deal with the existence of a solution to the viscoelastic dynamic model; in particular in Section~\ref{sub1}, we provide a solution by means of a generalization of Lax-Milgram's theorem by Lions. After that, in Section~\ref{sub2}, as previously anticipated, we define a system equivalent to the equation. In particular, in Subsection~\ref{sub3} we implement the time discretization method on such a system, and we conclude with Subsection~\ref{sub4} by showing the validity of the energy-dissipation inequality, and of the initial conditions.

\section{Notation}\label{visc_sec2}
In this section we fix some notation that will be used throughout the paper. The space of $m\times d$ matrices with real entries is denoted by $\R^{m\times d}$; in case $m=d$, the subspace of symmetric matrices is denoted by $\mathbb R^{d\times d}_{sym}$. Given a function $u\colon\R^d\to\R^m$, we denote its Jacobian matrix by $\nabla u$, whose components are $(\nabla u)_{ij}:= \partial_j u_i$ for $i=1,\dots,m$ and $j=1,\dots,d$; when $u\colon \R^d\to\R^d$, we use $eu$ to denote the symmetric part of the gradient, namely $eu:=\frac{1}{2}(\nabla u+\nabla u^T)$. Given a tensor field $A\colon \R^d\to\R^{m\times d}$, by $\div A$ we mean its divergence with respect to rows, namely $(\div A)_i:= \sum_{j=1}^d\partial_jA_{ij}$ for $i=1,\dots,m$. 

We denote the $d$-dimensional Lebesgue measure by $\mathcal L^d$ and the $(d-1)$-dimensional Hausdorff measure by $\mathcal H^{d-1}$; given a bounded open set $\Omega$ with Lipschitz boundary, by $\nu$ we mean the outer unit normal vector to $\partial\Omega$, which is defined $\mathcal H^{d-1}$-a.e. on the boundary. The Lebesgue and Sobolev spaces on $\Omega$ are defined as usual; the boundary values of a Sobolev function are always intended in the sense of traces. 

The norm of a generic Banach space $X$ is denoted by $\|\cdot\|_X$; when $X$ is a Hilbert space, we use $(\cdot,\cdot)_X$ to denote its scalar product. We denote by $X'$ the dual of $X$ and by $\spr{\cdot}{\cdot}_{X'}$ the duality product between $X'$ and $X$. Given two Banach spaces $X_1$ and $X_2$, the space of linear and continuous maps from $X_1$ to $X_2$ is denoted by $\mathscr L(X_1;X_2)$; given $\mathbb A\in\mathscr L(X_1;X_2)$ and $u\in X_1$, we write $\mathbb A u\in X_2$ to denote the image of $u$ under $\mathbb A$. 

Moreover, given an open interval $(a,b)\subset\R$ and $p\in[1,\infty]$, we denote by $L^p(a,b;X)$ the space of $L^p$ functions from $(a,b)$ to $X$; we use $H^k(a,b;X)$ to denote the Sobolev space of functions from $(a,b)$ to $X$ with $k$ derivatives in $L^2(a,b;X)$. Given $u\in H^1(a,b;X)$, we denote by $\dot u\in L^2(a,b;X)$ its derivative in the sense of distributions. When dealing with an element $u\in H^{1}(a,b;X)$ we always assume $u$ to be the continuous representative of its class, and therefore, the pointwise value $u(t)$ of $u$ is well defined for every $t\in[a,b]$. We use $C_w^0([a,b];X)$ to denote the set of weakly continuous functions from $[a,b]$ to $X$, namely, the collection of maps $u\colon [a,b]\to X$ such that $t\mapsto \spr{x'}{u(t)}_{X'}$ is continuous from $[a,b]$ to $\R$, for every $x'\in X'$. We adopt the notation $Lip([a,b];X)$ to denote the space of Lipschitz functions from the interval $[a,b]$ into the Banach space $X$.

\section{Formulation of the evolution problem, notion of solution}\label{visc_sec3M}
Let $T$ be a positive real number and $d\in\mathbb{N}$. Let $\Omega\subset\R^d$ be a bounded open set (which represents the reference configuration of the body) with Lipschitz boundary. Let $\partial_D\Omega$ be a (possibly empty) Borel subset of $\partial\Omega$, on which we prescribe the Dirichlet condition, and let $\partial_N\Omega$ be its complement, on which we give the Neumann condition. Let $\Gamma\subset\overline\Omega$ be the prescribed crack path. We assume the following hypotheses on the geometry of the cracks:
\begin{itemize}
\item[(E1)] $\Gamma$ is a closed set with $\mathcal L^d(\Gamma)=0$ and $\mathcal H^{d-1}(\Gamma\cap\partial\Omega)=0$;
\item[(E2)] for every $x\in\Gamma$ there exists an open neighborhood $U$ of $x$ in $\R^d$ such that $(U\cap\Omega)\setminus\Gamma$ is the union of two disjoint open sets $U^+$ and $U^-$ with Lipschitz boundary;
\item[(E3)] $\{\Gamma_t\}_{t\in(-\infty,T]}$ is a family of closed subsets of $\Gamma$ satisfying $\Gamma_s\subset\Gamma_t$ for every $-\infty< s\le t\le T$.
\end{itemize}
Notice that the set $\Gamma_t$ represents the crack at time $t$. Thanks to (E1)--(E3) the space $L^2(\OG_t;\R^d)$ coincides with $L^2(\Omega;\R^d)$ for every $t\in(-\infty,T]$. In particular, we can extend a function $u\in L^2(\OG_t;\R^d)$ to a function in $L^2(\Omega;\R^d)$ by setting $u=0$ on $\Gamma_t$. Since $\mathcal H^{d-1}(\Gamma\cap\partial\Omega)=0$ the trace of $u\in H^1(\OG;\R^d)$ is well defined on $\partial\Omega$. Indeed, we may find a finite number of open sets with Lipschitz boundary $U_j\subset\OG$, $j=1,\dots k$, such that $\partial\Omega\setminus \Gamma\subset\cup_{j=1}^k\partial U_j$. There exists a positive constant $C$, depending only on $\Omega$ and $\Gamma$, such that
\begin{equation}\label{eq:vis_traceM}
\norm{u}_{L^2(\partial\Omega;\R^d)}\le C\norm{u}_{H^1(\OG;\R^d)}\quad\text{for every }u\in H^1(\OG;\R^d).
\end{equation}
Similarly, we can find a finite number of open sets $V_j\subset\OG$, $j=1,\dots l$, with Lipschitz boundary, such that $\OG=\cup_{j=1}^l V_j$. By using the second Korn's inequality in each $V_j$ (see, e.g.,~\cite[Theorem 2.4]{OSY}) and taking the sum over $j$ we can find a positive constant $C_K$, depending only on $\Omega$ and $\Gamma$, such that 
\begin{equation}\label{eq:hyp_kornM}
\norm{\nabla u}_{L^2(\Omega;\R^{d\times d})}^2\le C_K(\norm{u}_{L^2(\Omega;\R^d)}^2+\norm{e u}_{L^2(\Omega;\R^{d\times d}_{sym})}^2)\quad\text{for every }u\in H^1(\OG;\R^d).
\end{equation}

We set $H:=L^2(\Omega;\R^d)$, $H^d_s:=L^2(\Omega;\R^{d\times d}_{sym})$, $H^N:=L^2(\partial_N\Omega;\R^d)$ and $H^D:=L^2(\partial_D\Omega;\R^d)$; the symbols $(\cdot,\cdot)$ and $\|\cdot\|$ denote the scalar product and the norm in $H$ or in $H^d_s$, according to the context. Moreover, we define the following spaces
\begin{equation*}
V:=H^1(\OG;\R^d)\quad \text{ and }\quad  V_t:= H^1(\OG_t;\R^d)\quad \text{for every $t\in (-\infty,T]$}.
\end{equation*}
Notice that in the definition of $V_t$ and $V$, we are considering only the distributional gradient of $u$ in $\OG_t$ and in $\OG$, respectively, and not the one in $\Omega$. Taking into account \eqref{eq:hyp_kornM}, we shall use on the set $V_t$ (and also on the set $V$) the equivalent norm
\begin{equation*}
\norm{u}_{V_t}:=(\norm{u}^2+\norm{e u}^2)^{\frac{1}{2}}\quad\text{for every }u\in V_t.
\end{equation*}
Furthermore, by~\eqref{eq:vis_traceM}, we can consider for every $t\in(-\infty,T]$ the set
\begin{equation*}
V_t^D:=\{u\in V_t:u=0\text{ on }\partial_D\Omega\},
\end{equation*}
which is a closed subspace of $V_t$. 

We assume that the elasticity and viscosity tensors $\A$ and $\B$ satisfy the following assumptions:
\begin{align}
    &\A,\B\in L^\infty(\O;\mathscr L(\R^{d\times d}_{sym};\R^{d\times d}_{sym})),\label{linfinity}\\
    &\A(x)\eta_1\cdot\eta_2=\eta_1\cdot \A(x)\eta_2,&&\hspace{-0.4cm}\B(x)\eta_1\cdot\eta_2=\eta_1\cdot \B(x)\eta_2&&\text{for a.e. $x\in\O$ and for every $\eta_1,\eta_2\in\R_{sym}^{d\times d}$},\label{CB1.5}\\
    &\A(x)\eta\cdot\eta\geq C_{\A}|\eta|^2,&&\hspace{-0.4cm}\B(x)\eta\cdot\eta\geq C_{\B}|\eta|^2&&\text{for a.e. $x\in\O$ and for every $\eta\in\R_{sym}^{d\times d}$,}\label{CB2M}
\end{align}
for some positive constants $C_{\A}$ and $C_{\B}$ independent of $x$, and the dot denotes the Euclidean scalar product of matrices.

Let $\beta$ a positive real number. We wish to study the following viscoelastic dynamic system
\begin{equation}\label{classic_model_inf}
    \ddot u(t)-\div((\A +\B) eu(t))+ \int_{-\infty}^t\frac{1}{\beta}\e^{-\frac{t-\tau}{\beta}}\div(\B e u(\tau))\de \tau=f(t)\quad\text{in $\OG_t$, $t\in(-\infty,T)$},
\end{equation}
together with the boundary conditions
\begin{alignat}{4}
&u(t)=z(t) && \qquad \text{on $\partial_D\Omega$}, & \quad t\in(-\infty,T),\label{boundary1_inf}\\
&\Big[\left(\A +\B\right) eu(t)-\int_{-\infty}^t\frac{1}{\beta}\e^{-\frac{t-\tau}{\beta}} \B e u(\tau)\de \tau\Big]\nu=N(t)  &&\qquad\text{on $\partial_N\Omega$}, & \quad t\in(-\infty,T),\\
&\Big[\left(\A +\B\right) eu(t)-\int_{-\infty}^t\frac{1}{\beta}\e^{-\frac{t-\tau}{\beta}} \B e u(\tau)\de \tau\Big]\nu=0  &&\qquad\text{on $\Gamma_t$}, & \quad t\in(-\infty,T),\label{boundary_inf}
\end{alignat}
where the data satisfy 
\begin{itemize}
    \item [(D1)] $f\in L^2_{loc}((-\infty;T];H)$;
    \item [(D2)] $N\in L^2_{loc}((-\infty;T];H^N)$ such that $\dot N\in L^2_{loc}((-\infty;T];H^N)$;
    \item [(D3)] $z\in L^2_{loc}((-\infty;T];V)$ such that $\dot z\in L^2_{loc}((-\infty;T];V)$, $\ddot z\in L^2_{loc}((-\infty;T];H)$, and $z(t)\in V_t$ for every $t\in(-\infty,T]$.
\end{itemize}
Notice that in \eqref{classic_model_inf}--\eqref{boundary_inf} the explicit dependence on $x$ is omitted to enlighten notation.

As usual, the Neumann boundary conditions are only formal, and their meaning will be specified in Definition~\ref{defwefirst_inf}. To this aim, we define $\V_{loc}(-\infty,T)$ as the space of all function $u\in L^2_{loc}((-\infty,T];V)$ such that $\dot u\in L^2_{loc}((-\infty,T];H)$, $u(t)\in V_t$ for a.e. $t\in (-\infty,T)$, and 
\begin{equation}\label{peso}
    \int_{-\infty}^T\e^{\frac{t}{\beta}}\|eu(t)\|\de t< +\infty.
\end{equation}

Now we are in position to explain in which sense we mean that $u\in\V_{loc}(-\infty,T)$ is a solution to the viscoelastic dynamic system \eqref{classic_model_inf}--\eqref{boundary_inf}. Roughly speaking, we multiply \eqref{classic_model_inf} by a test function, we integrate by parts in time and in space, and taking into account \eqref{boundary1_inf}--\eqref{boundary_inf} we obtain the following definition.

\begin{definition}[Weak solution]\label{defwefirst_inf}
We say that $u\in\V_{loc}(-\infty,T)$ is a {\it weak solution} to system \eqref{classic_model_inf} with boundary conditions \eqref{boundary1_inf}--\eqref{boundary_inf} if $u(t)-z(t)\in V_t^D$ for a.e. $t\in (-\infty,T)$, and 
\begin{align*}
-\int_{-\infty}^T(\dot u(t),\dot v(t)) \de t+\int_{-\infty}^T((\A +\B) eu(t),e v(t)) \de t&-\int_{-\infty}^T\int_{-\infty}^t\frac{1}{\beta}\e^{-\frac{t-\tau}{\beta}}(\B eu(\tau),e v(t)) \de\tau\de t\nonumber\\
&=\int_{-\infty}^T(f(t),v(t)) \de t+\int_{-\infty}^T(N(t),v(t))_{H^N} \de t
\end{align*}
for every $v \in C^{\infty}_c(-\infty,T;V)$ such that $v(t)\in V^D_t$ for every $t\in(-\infty,T]$.
\end{definition}

Now, let us consider $a,b\in[0,T]$ such that $a<b$. We define the spaces
\begin{align*}
\V(a,b)&:=\{u\in L^2(a,b;V)\cap H^1(a,b;H): u(t)\in V_t\hspace{2pt}\text{ for a.e. $t\in(a,b)$}\},\\
\V^D(a,b)&:=\{v\in\V(a,b):v(t)\in V_t^D\hspace{2pt} \text{for a.e. $t\in(a,b)$}\},\\
\D(a,b)&:=\{v\in C^{\infty}_c(a,b;V):v(t)\in V_t^D\hspace{2pt}\text{for every $t\in [a,b]$}\},
\end{align*}
and we have the following lemma.
\begin{lemma}\label{spazi_dens}
The space $\V(a,b)$ is a Hilbert space with respect to the following norm
\begin{equation*}
\norm{\varphi}_{\V(a,b)}:=\left(\norm{\varphi}_{L^2(a,b;V)}^2+\norm{\dot \varphi}_{L^2(a,b;H)}^2\right)^{\frac{1}{2}},\quad\varphi\in\V(a,b).
\end{equation*} 
Moreover, $\V^D(a,b)$ is a closed subspace of $\V(a,b)$, and $\D(a,b)$ is a dense subset of the space of functions belonging to $\V^D(a,b)$ which vanish on $a$ and $b$.
\end{lemma}
\begin{proof}
It is clear that $\norm{\cdot}_{\V(a,b)}$ is a norm induced by a scalar product on the set $\V(a,b)$. We just have to check the completeness of this space with respect to this norm. 
Let $\{\varphi_k\}_k\subset \V(a,b)$ be a Cauchy sequence. Then, $\{\varphi_k\}_k$ and $\{\dot \varphi_k\}_k$ are Cauchy sequences in $L^2(a,b;V)$ and $L^2(a,b;H)$, respectively, which are complete Hilbert spaces. Thus, there exists $\varphi\in L^2(a,b;V)$ with $\dot \varphi\in L^2(a,b;H)$ such that $\varphi_k\to \varphi$ in $L^2(a,b;V)$ and $\dot \varphi_k\to\dot \varphi$ in $L^2(a,b;H)$. In particular there exists a subsequence $\{\varphi_{k_j}\}_j$ such that $\varphi_{k_j}(t)\to \varphi(t)$ in $V$ for a.e. $t\in(a,b)$. Since $\varphi_{k_j}(t)\in V_t$ for a.e. $t\in(a,b)$ we deduce that $\varphi(t)\in V_t$ for a.e. $t\in(a,b)$. Hence $\varphi\in\V(a,b)$ and $\varphi_k\to \varphi$ in $\V(a,b)$. With a similar argument, we can prove that $\V^D(a,b)\subset \V(a,b)$ is a closed subspace. For the proof of the last statement we refer to~\cite[Lemma 2.8]{DMT2}.
\end{proof}

Now, suppose we know the past history of the system up to time $0$. In particular, let $u_p\in \V_{loc}(-\infty,0)$ be a weak solution to \eqref{classic_model_inf}--\eqref{boundary_inf} on the interval $(-\infty,0)$ in the sense of Definition \ref{defwefirst_inf}, in such a way that $0$ is a Lebesgue's point for both $u_p$ and $\dot u_p$. This implies that there exist $u^0\in V_0$, with $u^0-z(0)\in V_0^D$, and $u^1\in H$ such that
\begin{equation*}
    \lim_{h\to 0^+}\frac{1}{h}\int_{-h}^0\|u_p(t)-u^0\|^2_{V_0}\de t=0,\qquad \lim_{h\to 0^+}\frac{1}{h}\int_{-h}^0\|\dot u_p(t)-u^1\|^2\de t=0.
\end{equation*}
From this assumption, by defining
\begin{equation*}
    F_0(t):=\frac{1}{\beta}\e^{-\frac{t}{\beta}}\int_{-\infty}^0\e^{\frac{\tau}{\beta}}\B eu_p(\tau)\de \tau,
\end{equation*}
we can reformulate \eqref{classic_model_inf}--\eqref{boundary_inf} on the interval $[0,T]$ in the following way:
\begin{equation}\label{classic_model_inf2}
    \ddot u(t)-\div((\A +\B) eu(t))+ \int_{0}^t\frac{1}{\beta}\e^{-\frac{t-\tau}{\beta}}\div(\B e u(\tau))\de \tau=f(t)-\div F_0(t)\quad\text{in $\OG_t$, $t\in(0,T)$},
\end{equation}
with boundary and initial conditions
\begin{alignat}{4}
&u(t)=z(t) && \qquad \text{on $\partial_D\Omega$}, & \quad t\in(0,T),\label{boundary1_inf2}\\
&\Big[\left(\A +\B\right) eu(t)-\int_{0}^t\frac{1}{\beta}\e^{-\frac{t-\tau}{\beta}} \B e u(\tau)\de \tau\Big]\nu=N(t)+F_0(t)\nu  &&\qquad\text{on $\partial_N\Omega$}, & \quad t\in(0,T),\\
&\Big[\left(\A +\B\right) eu(t)-\int_{0}^t\frac{1}{\beta}\e^{-\frac{t-\tau}{\beta}} \B e u(\tau)\de \tau\Big]\nu=F_0(t)\nu  &&\qquad\text{on $\Gamma_t$}, & \quad t\in(0,T),\label{boundary_inf2}\\
&u(0)=u^0,\quad \dot u(0)=u^1.\label{sys_inf2}
\end{alignat}

Thanks to (D1)--(D3) and \eqref{peso} (on the interval $(-\infty,0]$), we have $f\in L^2(0,T;H)$, $F_0\in C^{\infty}([0,T];H^d_s)$, $N\in H^1(0,T;H^N)$, and $z\in H^2(0,T;H)\cap H^1(0,T;V)$ with $z(t)\in V_t$ for every $t\in[0,T]$. 

More in general, given $F\in H^1(0,T;H^d_s)$ we will study the following viscoelastic dynamic system
\begin{equation}\label{classic_model_inf3}
    \ddot u(t)-\div((\A +\B) eu(t))+ \int_{0}^t\frac{1}{\beta}\e^{-\frac{t-\tau}{\beta}}\div(\B e u(\tau))\de \tau=f(t)-\div F(t)\quad\text{in $\OG_t$, $t\in(0,T)$},
\end{equation}
with boundary and initial conditions
\begin{alignat}{4}
&u(t)=z(t) && \qquad \text{on $\partial_D\Omega$}, & \quad t\in(0,T),\label{boundary1_inf3}\\
&\Big[\left(\A +\B\right) eu(t)-\int_{0}^t\frac{1}{\beta}\e^{-\frac{t-\tau}{\beta}} \B e u(\tau)\de \tau\Big]\nu=F(t)\nu  &&\qquad\text{on $\partial_N\Omega$}, & \quad t\in(0,T),\\
&\Big[\left(\A +\B\right) eu(t)-\int_{0}^t\frac{1}{\beta}\e^{-\frac{t-\tau}{\beta}} \B e u(\tau)\de \tau\Big]\nu=F(t)\nu  &&\qquad\text{on $\Gamma_t$}, & \quad t\in(0,T),\label{boundary_inf3}\\
&u(0)=u^0,\quad \dot u(0)=u^1.\label{sys_inf3}
\end{alignat}

Notice that system \eqref{classic_model_inf2}--\eqref{sys_inf2} is a particular case of system \eqref{classic_model_inf3}--\eqref{sys_inf3}. As we have already specified for system \eqref{classic_model_inf}--\eqref{boundary_inf}, also for \eqref{classic_model_inf3}--\eqref{sys_inf3} the Neumann boundary conditions are only formal, and their meaning is clarified by the following definition.

\begin{definition}\label{newdef2}
We say that $u\in\V(0,T)$ is a {\it weak solution} to the viscoelastic dynamic system \eqref{classic_model_inf3}--\eqref{sys_inf3} on the interval $[0,T]$ if $u-z\in \V^D(0,T)$,  
\begin{align}\label{first_weak}
-\int_0^T(\dot u(t),\dot v(t)) \de t+\int_0^T((\A +\B) eu(t),e v(t)) \de t&-\int_0^T\int_0^t\frac{1}{\beta}\e^{-\frac{t-\tau}{\beta}}(\B eu(\tau),e v(t)) \de\tau\de t\nonumber\\
&=\int_0^T(f(t),v(t)) \de t+\int_0^T(F(t),ev(t)) \de t
\end{align}
for every $v \in \D(0,T)$, and 
\begin{equation}\label{ini*}
    \lim_{t\to 0^+}\|u(t)-u^0\|=0,\quad\lim_{t\to 0^+}\|\dot u(t)-u^1\|_{(V^D_0)'}=0.
 \end{equation}
\end{definition}

\begin{remark}\label{dens}
From Lemma \ref{spazi_dens}, if a function $u\in\V(0,T)$ satisfies \eqref{first_weak} for every $v\in \D(0,T)$, then it satisfies the same equality for every $v\in \V^D(0,T)$ such that $v(0)=v(T)=0$.
\end{remark}

\section{Existence by using Dafermos' method}\label{sub1}
 In this section we present an existence result which is to be considered in the framework of functional analysis; in particular it derives from an idea of C. Dafermos (see~\cite{DA}) based on a generalization of Lax-Milgram's Theorem by J.L. Lions (see~\cite{LL}). We start by stating the main result of this section.
 \begin{theorem}\label{main1}
 There exists a weak solution $u\in \V(0,T)$ to the viscoelastic dynamic system \eqref{classic_model_inf3}--\eqref{sys_inf3} on the interval $[0,T]$ in the sense of Definition \ref{newdef2}. Moreover, there exists a positive constant $C=C(T,\A,\B,\beta)$ such that
 \begin{equation}\label{stima_sol}
     \|u\|_{\V(0,T)}\leq C\left(\|f\|_{L^2(0,T;H)}+\|F\|_{H^1(0,T;H^d_s)}+\|\ddot z\|_{L^2(0,T;H)}+\| z\|_{H^1(0,T;V)}+\|u^0\|_V+\|u^1\|\right).
 \end{equation}
 \end{theorem}
 
 \begin{remark}\label{remi}
Without loss of generality we may assume that the Dirichlet datum and the initial displacement are identically equal to zero. Indeed, the function $u$ is a weak solution to the viscoelastic dynamic system \eqref{classic_model_inf3}--\eqref{sys_inf3} according to Definition \ref{newdef2} if and only if the function $u^*$ defined by 
$$u^*(t):=u(t)-u^0+z(0)-z(t),$$
satisfies
    \begin{align*}
        -\int_0^T(\dot {u}^*(t),\dot \psi(t)) \de t+\int_0^T((\A +\B) e u^*(t),e \psi(t)) \de t&-\int_0^T\int_0^t\frac{1}{\beta}\e^{-\frac{t-\tau}{\beta}}(\B  e u^*(\tau),e \psi(t)) \de\tau\de t\nonumber\\
&=\int_0^T( f^* (t),\psi(t)) \de t+\int_0^T(  F^*(t),e\psi(t))\de t,
    \end{align*}
for every $\psi\in\D(0,T) $, and
\begin{align*}
        \lim_{t\to 0^+}\|  u^*(t)\|=0,\quad\lim_{t\to 0^+}\|\dot { u}^*(t)-u^1_*\|_{(V^D_0)'}=0,
    \end{align*}
where $  f^* :=f-\ddot z$, $ u^1_*:=u^1-\dot z(0)$, and for every $t\in [0,T]$ 
\begin{align*}
 F^*(t):=F(t)+\int_0^t\frac{1}{\beta}\e^{-\frac{t-\tau}{\beta}}\B ez(\tau)\de \tau-(\A +\B)  ez(t)- (\A+\e^{-\frac{t}{\beta}}\B )(eu^0-ez(0)).
\end{align*}

Moreover, if $u^*$ satisfies for some positive constants $C^*$ the following estimate
\begin{equation}\label{stima_par}
     \|u^*\|_{\V(0,T)}\leq C^*\left(\|f^*\|_{L^2(0,T;H)}+\|F^*\|_{H^1(0,T;H^d_s)}+\|u^1_*\|\right),
\end{equation}
then $u$ satisfies \eqref{stima_sol}. Indeed, since
\begin{align*}
    \|f^*\|_{L^2(0,T;H)}&\leq  \|f\|_{L^2(0,T;H)}+ \|\ddot z\|_{L^2(0,T;H)},
\end{align*}
and for some positive constants $\bar{C}=C(T,\A,\B,\beta)$ we have
\begin{align*}
     \|F^*\|_{H^1(0,T;H^d_s)}&\leq \|F\|_{H^1(0,T;H^d_s)}+\Big(1+\frac{2^{\frac{1}{2}}}{\beta}\Big)\Big\|\int_0^{\cdot}\frac{1}{\beta}\e^{-\frac{\cdot-\tau}{\beta}}\B ez(\tau)\de \tau\Big\|_{L^2(0,T;H^d_s)}+\frac{2^{\frac{1}{2}}}{\beta}\|\B\|_{\infty}\|z\|_{L^2(0,T;V)}\nonumber\\
     &\hspace{0.5cm}+\|\A+\B\|_{\infty}\|z\|_{H^1(0,T;V)}+(\|\A\|_{\infty}+\|\e^{-\frac{\cdot}{\beta}}\|_{H^1(0,T)}\|\B\|_{\infty})(\|u^0\|_V+\|z(0)\|_V)\nonumber\\
     &\leq \bar{C}(\|F\|_{H^1(0,T;H^d_s)}+\|z\|_{H^1(0,T;V)}+\|u^0\|_V),
\end{align*}
from \eqref{stima_par} we deduce
\begin{align*}
    \|u\|_{\V(0,T)}&\leq \|u^*\|_{\V(0,T)}+T^{\frac{1}{2}}(\|u^0\|_V+\|z(0)\|_V)+\|z\|_{\V(0,T)}\\
    &\leq C\left(\|f\|_{L^2(0,T;H)}+\|F\|_{H^1(0,T;H^d_s)}+\|\ddot z\|_{L^2(0,T;H)}+\| z\|_{H^1(0,T;V)}+\|u^0\|_V+\|u^1\|\right),
\end{align*}
where $C=C(T,\A,\B,\beta)$ is a positive constant.
\end{remark}

Based on Remark \ref{remi}, we now assume that the Dirichlet datum and the initial displacement are identically equal to zero. To prove the theorem in this case, we first prove that our weak formulation \eqref{first_weak} with initial conditions \eqref{ini*} is equivalent to another one, which we call Dafermos' Equality. After that, by means of a Lions' theorem we prove that there exists an element which satisfies this equality. Namely, by defining for every $a,b\in[0,T]$ such that $a<b$ the space
\begin{equation*}
    \E^D_0(a,b):=\{\varphi\in C^{\infty}([a,b];V):\varphi(a)=0, \hspace{2pt}\varphi(t)\in V^D_t\hspace{2pt} \text{for every $t\in [a,b]$}\},
\end{equation*}
we can state the following equivalence result.
\begin{proposition}\label{daf}
Suppose that there exists $u\in \V^D(0,T)$ which satisfies the initial condition $u(0)=0$ in the sense of \eqref{ini*}, and such that Dafermos' Equality holds:
\begin{align}\label{dafermos2}
    \int_0^T(\dot u(t),\dot \varphi (t))\de t&\hspace{-1pt}+\hspace{-1pt}\int_0^T (t-T)\Big[(\dot u(t),\ddot \varphi(t))\hspace{-1pt}-\hspace{-1pt}((\A +\B) eu(t),e\dot\varphi(t))\hspace{-1pt}+\hspace{-1pt}\int_0^t\frac{1}{\beta}\e^{-\frac{t-\tau}{\beta}}(\B eu(\tau),e \dot \varphi(t))\de \tau\Big]\hspace{-1pt}\de t\nonumber\\
    &=T(u^1,\dot \varphi(0))-\int_0^T(t-T)\left[( f (t),\dot \varphi(t))+( F(t), e\dot\varphi(t))\right] \de t\quad \text{for every $\varphi\in \E_0^D(0,T)$}.
\end{align}
Then $u$ satisfies \eqref{first_weak}, $u(0)=0$ and $\dot u(0)$ coincides with $u^1$ in $(V^D_0)'$. Moreover, if $u\in \V^D(0,T)$ is a weak solution in the sense of Definition \ref{newdef2}, then it satisfies \eqref{dafermos2}.
\end{proposition}

At this point, we state and prove some lemmas and propositions needed for the proof of Proposition \ref{daf}. In particular, in the following lemma, we highlight a useful relation between $\D(0,T)$ and $\E_0^D(0,T)$. 
\begin{lemma}\label{spaziM}
For every $v\in \D(0,T)$ the function defined by
\begin{equation*}
    \varphi_v(t)=\int_0^t \frac{v(\tau)}{\tau-T}\de \tau
\end{equation*}
is well defined and satisfies $\varphi_v\in \E^D_0(0,T)$.
\end{lemma}
\begin{proof}
Firstly, we can notice that $\varphi_v$ is well defined because $v$ is a function with compact support, hence it vanishes in a neighborhood of $T$. Moreover, $\varphi_v(0)=0$ by definition and $\varphi_v\in C^{\infty}([0,T];V)$ because it is a primitive of a function with the same regularity. Now, we can observe that $v(\tau)\in V^D_{\tau}\subset V^D_t$ for every $\tau\leq t$, therefore we have $\frac{v(\tau)}{\tau-T}\in V^D_t$ for every $\tau\leq t$, and by the properties of Bochner's integral we get $\varphi_v(t)\in V^D_t$. 
\end{proof}

In the next proposition we show that the distributional second derivative in time of a weak solution is an element of the space $L^2(0,T;(V_0^D)')$. Therefore, such a solution has an initial velocity in the space $ (V_0 ^ D) '$.
\begin{proposition}\label{dev2}
Let $u\in \V^D(0,T)$ be a function which satisfies \eqref{first_weak}. Then the distributional derivative of $\dot u$ belongs to the space $L^2(0,T;(V^D_0)') $. 
\end{proposition}
\begin{proof}
Let $\Lambda\in L^2(0,T;(V^D_0)')$ be defined in the following way: for a.e. $t\in (0,T)$
\begin{equation}\label{defop}
    \langle \Lambda(t),\vs\rangle:=-((\A +\B)  eu(t),e\vs)+\int_0^t\frac{1}{\beta}\e^{-\frac{t-\tau}{\beta}}(\B e u(\tau),e\vs)\de \tau+( f (t),\vs)+( F(t),e\vs)\quad \text{for every $\vs\in V^D_0$,}
\end{equation}
where $\langle \cdot, \cdot \rangle$ represents the duality product between $(V^D_0)'$ and $V^D_0$. \\
Let us consider a test function $\varphi\in C^{\infty}_c(0,T)$, then for every $\vs\in V^D_0$ the function $\psi(t):=\varphi(t)\vs$ belongs to the space $C^{\infty}_c(0,T;V_0)$, and consequently $\psi\in\D(0,T)$. Now we multiply both sides of \eqref{defop} by $\varphi(t)$ and we integrate it on $(0,T)$. Thanks to \eqref{first_weak} we can write
\begin{align*}
    \int_0^T \langle \Lambda(t),\vs\rangle\varphi(t)\de t=&-\int_0^T((\A +\B)  eu(t),e\psi(t))\de t+\int_0^T\int_0^t\frac{1}{\beta}\e^{-\frac{t-\tau}{\beta}}(\B e u(\tau),e\psi(t))\de \tau\de t\nonumber\\
    &+\int_0^T( f (t),\psi(t))\de t+\int_0^T( F(t),e\psi(t))\de t=-\int_0^T(\dot u(t),\vs)\dot\varphi(t)\de t,
\end{align*}
which implies
\begin{equation*}
   \Big\langle \int_0^T \Lambda(t)\varphi(t)\de t,\textbf{v}\Big\rangle=\Big\langle-\int_0^T\dot u(t)\dot\varphi(t)\de t,\textbf{v}\Big\rangle\qquad \text{for every $\vs\in V^D_0$}.
\end{equation*}
Hence, we get
\begin{equation*}
 \int_0^T \Lambda(t)\varphi(t)\de t=-\int_0^T\dot u(t)\dot\varphi(t)\de t\qquad \text{for every $\varphi\in C^{\infty}_c(0,T)$}
\end{equation*}
as elements of $(V^D_0)'$, which concludes the proof.
\end{proof}
\begin{remark}\label{udotcon}
Proposition \ref{dev2} implies that $\dot u\in H^1(0,T;(V^D_0)')$, hence it admits a continuous representative. Therefore, we can say that there exists $\dot u(0)\in (V^D_0)'$ such that 
\begin{equation}\label{velocity}
    \lim_{t\to 0^+}\|\dot u(t)-\dot u(0)\|_{(V^D_0)'}=0.
\end{equation}
\end{remark}
In the next proposition we show how the weak formulation \eqref{first_weak} changes if we use test functions which do not vanish at zero. In particular, we use the notation $\eta(T)$ to refer to the family of open neighborhoods of $T$, and we consider the following spaces
\begin{align*}
\Li&:=\{\psi\in Lip([0,T];V): \psi(t)\in V^D_t\hspace{2pt}\text{for every $t\in [0,T]$}\},\\
\L&:=\{\psi\in \Li:\exists I_{\psi}\in \eta(T), \hspace{2pt}\text{s.t.}\hspace{4pt}\psi(t)=0 \quad\text{for every $t\in I_{\psi}\cup\{0\}$}\},\\
\LT&:=\{\Psi\in \Li:\Psi(T)=0\}.
\end{align*}

\begin{proposition}\label{pass}
Let $u\in \V^D(0,T)$ be a function which satisfies \eqref{first_weak} for every $\psi\in \L$. Then $u$ satisfies the equality
\begin{align}
-\int_0^T(\dot u(t),\dot \Psi(t)) \de t+\int_0^T((\A &+\B) eu(t),e \Psi(t)) \de t-\int_0^T\int_0^t\frac{1}{\beta}\e^{-\frac{t-\tau}{\beta}}(\B eu(\tau),e \Psi(t)) \de\tau\de t\nonumber\\
&=\int_0^T( f (t),\Psi(t)) \de t+\int_0^T( F(t),e\Psi (t)) \de t+\langle \dot u(0),\Psi(0)\rangle, \label{modM}
\end{align}
for every $\Psi\in \LT$.
\end{proposition}
\begin{proof}
Let us consider $\Psi\in \LT$ and define for every $\varepsilon\in(0,\frac{T}{3})$ the function
\begin{equation*}
    \psi_{\varepsilon}(t):=\begin{cases}
    \frac{t}{\varepsilon}\Psi(0) &t\in [0,\varepsilon]\\
    \Psi(t-\varepsilon) &t\in[\varepsilon,T-2\varepsilon]\\
    \left(-\frac{t}{\varepsilon}+\frac{T-\varepsilon}{\varepsilon}\right)\Psi(T-3\varepsilon) &t\in[T-2\varepsilon,T-\varepsilon]\\
    0 &t\in[T-\varepsilon,T].\\
    \end{cases}
\end{equation*}
It is easy to see that $\psi_{\varepsilon}\in \L$, and by using $\psi_{\varepsilon}$ as test function in \eqref{first_weak} we get $I_{\varepsilon}+I^m_{\varepsilon}+J_{\varepsilon}^m=0$, where the three terms $I_{\varepsilon}$, $I^m_{\varepsilon}$, and $J_{\varepsilon}^m$ are defined in the following way:
\begin{align*}
    I_{\varepsilon}:=&-\int_{\varepsilon}^{T-2\varepsilon}(\dot u(t),\dot \Psi(t-\varepsilon))\de t+\int_{\varepsilon}^{T-2\varepsilon}((\A +\B)  eu(t),e\Psi(t-\varepsilon))\de t-\int_{\varepsilon}^{T-2\varepsilon}( f (t),\Psi(t-\varepsilon))\de t\\
    &-\int_{\varepsilon}^{T-2\varepsilon}\int_0^t\frac{1}{\beta}\e^{-\frac{t-\tau}{\beta}}(\mathbb B eu(\tau),e\Psi(t-\varepsilon))\de \tau\de t-\int_{\varepsilon}^{T-2\varepsilon}( F(t),e\Psi(t-\varepsilon))\de t,
    \end{align*}
\begin{align*}
   I_{\varepsilon}^m:=-\dashint_0^{\varepsilon}(\dot u(t),\Psi(0))\de t+\dashint_0^{\varepsilon}((\A +\B)  eu(t),te\Psi(0))\de t&-\dashint_0^{\varepsilon}\int_0^t\frac{1}{\beta}\e^{-\frac{t-\tau}{\beta}}(\mathbb B eu(\tau),te\Psi(0))\de\tau\de t\\
   &-\dashint_0^{\varepsilon}( f (t),t\Psi(0))\de t-\dashint_0^{\varepsilon}(  F(t),t e\Psi(0))\de t,
\end{align*}
and
\begin{align*}
   J_{\varepsilon}^m:&=\dashint_{T-2\varepsilon}^{T-\varepsilon}(\dot u(t),\Psi(T-3\varepsilon))\de t+\dashint_{T-2\varepsilon}^{T-\varepsilon}((\A +\B)  eu(t),(-t+T-\varepsilon)e\Psi(T-3\varepsilon))\de t\\
   &-\dashint_{T-2\varepsilon}^{T-\varepsilon}\int_0^t\frac{1}{\beta}\e^{-\frac{t-\tau}{\beta}}(\mathbb B eu(\tau),(-t+T-\varepsilon)e\Psi(T-3\varepsilon))\de\tau\de t-\dashint_{T-2\varepsilon}^{T-\varepsilon}( f (t),(-t+T-\varepsilon)\Psi(T-3\varepsilon))\de t\\
   &-\dashint_{T-2\varepsilon}^{T-\varepsilon}(  F(t),(-t+T-\varepsilon) e\Psi(T-3\varepsilon))\de t.
\end{align*}
Let us study the convergence of $I_{\varepsilon}$, $I_{\varepsilon}^m$, and $J_{\varepsilon}^m$  as $\varepsilon\rightarrow 0^+$. First of all, we notice that from the definition of $\psi_{\varepsilon}$ and the Lipschitz continuity of $\Psi$ we have
\begin{align}\label{this}
    \|\psi_{\varepsilon}-\Psi\|^2_{L^2(0,T;V)}&=\int_0^{\varepsilon} \Big\|\frac{t}{\varepsilon}\Psi(0)-\Psi(t) \Big\|^2_{V}\de t+\int_{\varepsilon}^{T-2\varepsilon} \|\Psi(t-\varepsilon)-\Psi(t)\|^2_{V}\de t\nonumber\\
    &\hspace{4,05cm}+\int_{T-2\varepsilon}^{T-\varepsilon} \Big\|\Big(-\frac{t}{\varepsilon}+\frac{T-\varepsilon}{\varepsilon}\Big)\Psi(T-3\varepsilon)-\Psi(t)\Big\|^2_{V}\de t\nonumber\\
    &\leq 2\|\Psi(0)\|^2_{V}\int_0^{\varepsilon}\frac{t^2}{\varepsilon^2}\de t+2\int_0^{\varepsilon}\|\Psi(t)\|^2_{V}\de t+\int_{\varepsilon}^{T-2\varepsilon} L^2_{\Psi}|t-\varepsilon-t|^2\de t\nonumber\\
    &\hspace{3,3cm}+2\|\Psi(T-3\varepsilon)\|^2_{V}\int_{T-2\varepsilon}^{T-\varepsilon}\Big(-\frac{t}{\varepsilon}+\frac{T-\varepsilon}{\varepsilon}\Big)^2\de t+2\int_{T-2\varepsilon}^{T-\varepsilon}\|\Psi(t)\|^2_{V}\de t\nonumber\\
    &\leq\frac{4}{3}\varepsilon \|\Psi\|^2_{L^{\infty}(0,T;V)}+2\int_0^{\varepsilon}\|\Psi(t)\|^2_{V}\de t+2\int_{T-2\varepsilon}^{T-\varepsilon}\|\Psi(t)\|^2_{V}\de t+ L^2_{\Psi}\varepsilon^2(T-3\varepsilon)\xrightarrow[\varepsilon\to 0^+]{}0.
\end{align}
From \eqref{linfinity}, \eqref{this}, and the absolute continuity of Lebesgue's integral, we have
\begin{align}
    \Big|\int_{\varepsilon}^{T-2\varepsilon}((\A &+\B)  eu(t),e\Psi(t-\varepsilon))\de t-\int_{0}^T((\A +\B)  eu(t),e\Psi(t))\de t\Big|\leq \Big|\int_0^{\varepsilon}((\A +\B)  eu(t),e\Psi(t))\de t\Big|\nonumber\\
    &+\Big|\int_{\varepsilon}^{T-2\varepsilon}((\A +\B)  eu(t),e\Psi(t-\varepsilon)-e\Psi(t))\de t\Big|+\Big|\int_{T-2\varepsilon}^{T}((\A +\B)  eu(t),e\Psi(t))\de t\Big|\nonumber \\ 
     &\leq \|\A+\B\|_{\infty}\Big[\int_0^{\varepsilon}\|u(t)\|_V\|\Psi(t)\|_V\de t+\int_{T-2\varepsilon}^{T}\|u(t)\|_V\|\Psi(t))\|_V\de t \nonumber\\
    &\hspace{5.6cm}+\|u\|_{L^2(0,T;V)}\|\psi_{\varepsilon}-\Psi\|_{L^2(0,T;V)}\Big]\xrightarrow[\varepsilon\to 0^+]{}0. \label{prima}
   \end{align}
In the same way we can prove that
\begin{align}
    \int_{\varepsilon}^{T-2\varepsilon}\int_0^t\frac{1}{\beta}\e^{-\frac{t-\tau}{\beta}}(\mathbb B eu(\tau),e\Psi(t-\varepsilon))\de \tau\de t&\xrightarrow[\varepsilon\to 0^+]{}\int_{0}^T\int_0^t\frac{1}{\beta}\e^{-\frac{t-\tau}{\beta}}(\mathbb B eu(\tau),e\Psi(t))\de \tau\de t,\\
    \int_{\varepsilon}^{T-2\varepsilon}( f (t),\Psi(t-\varepsilon))\de t&\xrightarrow[\varepsilon\to 0^+]{}\int_{0}^T( f (t),\Psi(t))\de t,\\
    \int_{\varepsilon}^{T-2\varepsilon}( F(t),e\Psi(t-\varepsilon))\de t&\xrightarrow[\varepsilon\to 0^+]{}\int_{0}^T( F(t),e\Psi(t))\de t.
\end{align}
Notice that, by virtue of the continuity of the translation operator in $L^2$, and again by the absolute continuity of Lebesgue's integral, we can write
\begin{align}\label{ultima}
   &\Big| \int_{\varepsilon}^{T-2\varepsilon}(\dot u(t),\dot \Psi(t-\varepsilon))\de t-\int_{0}^T(\dot u(t),\dot \Psi(t))\de t\Big|\nonumber\\
   &\hspace{1cm}\leq \Big|\int_0^{\varepsilon}(\dot u(t),\dot \Psi(t))\de t\Big|+\Big|\int_{\varepsilon}^{T-2\varepsilon}(\dot u(t),\dot \Psi(t-\varepsilon)-\dot \Psi(t))\de t\Big|+\Big|\int_{T-2\varepsilon}^{T}(\dot u(t),\dot \Psi(t))\de t\Big|\nonumber\\
& \hspace{1cm}\leq\int_0^{\varepsilon}\|\dot u(t)\|\|\dot \Psi(t)\|\de t+\|\dot u\|_{L^2(0,T;H)}\|\dot \Psi(\cdot-\varepsilon)-\dot \Psi(\cdot)\|_{L^2(0,T;H)}+\int_{T-2\varepsilon}^{T}\|\dot u(t)\|\|\dot \Psi(t))\|\de t\xrightarrow[\varepsilon\to 0^+]{}0.
\end{align}
Taking into account \eqref{prima}--\eqref{ultima} we conclude that
\begin{align*}
    I_{\varepsilon}\xrightarrow[\varepsilon\to 0^+]{}-\int_{0}^T(\dot u(t),\dot\Psi(t))\de t+\int_{0}^T((\A +\B)  eu(t),e\Psi(t))\de t&-\int_{0}^T\int_0^t\frac{1}{\beta}\e^{-\frac{t-\tau}{\beta}}(\mathbb B eu(\tau),e\Psi(t))\de \tau\de t\nonumber\\
    &\hspace{-3pt}-\int_{0}^T( f (t),\Psi(t))\de t-\int_{0}^T( F(t),e\Psi(t))\de t.
\end{align*}

Now we analyze the limit of $I^m_{\varepsilon}$ as $\varepsilon\to 0^+$. By \eqref{velocity} we obtain 
\begin{align}\label{prima*M}
    \dashint_0^{\varepsilon}(\dot u(t),\Psi(0))\de t=\Big( \dashint_0^{\varepsilon}\dot u(t)\de t,\Psi(0)\Big)=\Big\langle  \dashint_0^{\varepsilon}\dot u(t)\de t,\Psi(0)\Big\rangle\xrightarrow[\varepsilon\to 0^+]{}\langle \dot u(0),\Psi(0)\rangle.
\end{align}
Moreover
\begin{align}
    \Big|\dashint_0^{\varepsilon}((\A +\B)  eu(t),te\Psi(0))\de t\Big|&\leq \|\A +\B \|_{{\infty}}\|\Psi(0)\|_V\dashint_0^{\varepsilon}t\|u(t)\|_{V}\de t\nonumber\\
    &\leq \|\A +\B \|_{{\infty}}\|\Psi\|_{L^{\infty}(0,T;V)} \Big(\frac{\varepsilon}{3}\Big)^{\frac{1}{2}}\|u\|_{L^2(0,T;V)}\xrightarrow[\varepsilon\to 0^+]{}0.
\end{align}
In the same way, we can prove that
\begin{align}
    &\dashint_0^{\varepsilon}\int_0^t\frac{1}{\beta}\e^{-\frac{t-\tau}{\beta}}(\mathbb B eu(\tau),te\Psi(0))\de\tau\de t\xrightarrow[\varepsilon\to 0^+]{}0,\\
    &\dashint_0^{\varepsilon}( f (t),t\Psi(0))\de t\xrightarrow[\varepsilon\to 0^+]{}0,\\
    &\dashint_0^{\varepsilon}( F(t),te\Psi(0))\de t\xrightarrow[\varepsilon\to 0^+]{}0,\label{ultima*M}
\end{align}
hence, by \eqref{prima*M}--\eqref{ultima*M} we obtain $I^m_{\varepsilon}\xrightarrow[\varepsilon\to 0^+]{}-\langle \dot u(0),\Psi(0)\rangle$. 

Finally, we study the behaviour of $J_{\varepsilon}^m$ as $\varepsilon\to 0^+$. Since $\Psi(T)=0$, we can write
\begin{equation}\label{ultima_new}
   \Big| \dashint_{T-2\varepsilon}^{T-\varepsilon}(\dot u(t),\Psi(T-3\varepsilon))\de t\Big|\leq \frac{1}{\varepsilon^{\frac{1}{2}}}\|\dot u\|_{L^2(0,T;H)}\|\Psi(T-3\varepsilon)-\Psi(T)\|\leq 3L_{\Psi}\|\dot u\|_{L^2(0,T;H)}\varepsilon^{\frac{1}{2}}\xrightarrow[\varepsilon\to 0^+]{}0.
\end{equation}
Moreover
\begin{align}\label{ultima_new*}
    \Big|\dashint_{T-2\varepsilon}^{T-\varepsilon}((\A +\B)  eu(t)&,(-t+T-\varepsilon)e\Psi(T-3\varepsilon))\de t\Big|\nonumber\\
    &\leq \|\A+\B\|_{\infty}\|\Psi(T-3\varepsilon)\|_{V}\Big(\dashint_{T-2\varepsilon}^{T-\varepsilon}(T-t)\|u(t)\|_V\de t+\int_{T-2\varepsilon}^{T-\varepsilon}\|u(t)\|_V\de t\Big)\nonumber\\
    &\leq \|\A+\B\|_{\infty}\|\Psi\|_{L^{\infty}(0,T;V)}\Big(\Big(\frac{7}{3}\Big)^{\frac{1}{2}}+1\Big)\varepsilon^{\frac{1}{2}}\|u\|_{L^2(0,T;V)}\xrightarrow[\varepsilon\to 0^+]{}0.
\end{align} 
By following the same strategy used in \eqref{ultima_new*}, we can prove that
\begin{align}
   &\dashint_{T-2\varepsilon}^{T-\varepsilon}\int_0^t\frac{1}{\beta}\e^{-\frac{t-\tau}{\beta}}(\mathbb B eu(\tau),(-t+T-\varepsilon)e\Psi(T-3\varepsilon))\de\tau\de t\xrightarrow[\varepsilon\to 0^+]{}0,\\
   &\dashint_{T-2\varepsilon}^{T-\varepsilon}( f (t),(-t+T-\varepsilon)\Psi(T-3\varepsilon))\de t\xrightarrow[\varepsilon\to 0^+]{}0,\\
   &\dashint_{T-2\varepsilon}^{T-\varepsilon}(  F(t),(-t+T-\varepsilon) e\Psi(T-3\varepsilon))\de t\xrightarrow[\varepsilon\to 0^+]{}0.\label{ultima_new**}
\end{align}
Thanks to \eqref{ultima_new}--\eqref{ultima_new**} we can say that $J_{\varepsilon}^m\rightarrow 0$ as $\varepsilon\to 0^+$, and this concludes the proof.
\end{proof}
We can now prove the equivalence result between the viscoelastic dynamic system~\eqref{classic_model_inf3}--\eqref{sys_inf3} (in the sense of Definition \ref{newdef2}) and Dafermos' Equality \eqref{dafermos2}, stated in Proposition \ref{daf}. 
\begin{proof}[Proof of Proposition \ref{daf}]
Let $u\in\V^D(0,T)$ be a function with $u(0)=0$, and which satisfies \eqref{dafermos2}. Let us consider $v\in \D(0,T)$. By Lemma \ref{spaziM}, the function defined by
\begin{equation}\label{phiv}
    \varphi_v(t)=\int_0^t \frac{v(\tau)}{\tau-T}\de \tau
\end{equation}
is well defined and belongs to the space $\E^D_0(0,T)$. By taking $\varphi_v$ as a test function in \eqref{dafermos2} we obtain
\begin{align}
    -\int_0^T(\dot u(t),\dot \varphi_v (t)+(t-T)\ddot\varphi_v(t))& \de t+\int_0^T\Big((\A +\B) eu(t)-\int_0^t\frac{1}{\beta}\e^{-\frac{t-\tau}{\beta}}\B eu(\tau)\de\tau,e ((t-T)\dot \varphi_v(t))\Big)\de t\nonumber\\
    &=\int_0^T(  f (t),(t-T)\dot \varphi_v(t))\de t+\int_0^T( F(t),e((t-T)\dot \varphi_v(t)))\de t\label{deb},
\end{align}
since $\dot \varphi_v(0)=\frac{v(0)}{-T}=0$. Notice that $v(t)=(t-T)\dot\varphi_v(t)$ and consequently $\dot v(t)=\dot \varphi_v(t)+(t-T)\ddot \varphi_v(t)$, by the definition of $\varphi_v$ itself. This, together with \eqref{deb}, allows us to conclude that $u\in \V^D(0,T)$ satisfies \eqref{first_weak} for every $v\in \D(0,T)$.

Now we prove that $u^1$ coincides with $\dot u(0)$. Since the function $u$ satisfies \eqref{first_weak} for every $v\in \D(0,T)$, in particular, from Remark \ref{dens}, it satisfies the same equality for every $v\in \L$. Thanks to Proposition \ref{pass}, the function $u$ satisfies \eqref{modM} for every $v\in \LT$, and therefore for every function in the space
\begin{equation*}
    \E^D_T(0,T):=\{v\in C^{\infty}([0,T];V):\exists I_v\in \eta(T), \hspace{3pt}\text{s.t.}\hspace{4pt} v(t)=0\hspace{3pt}\text{for every $t\in I_v$},\hspace{2pt} v(t)\in V^D_t\hspace{3pt}\text{for every $t\in [0,T]$}\}.
\end{equation*}
Moreover, if we define $\varphi_v$ as in \eqref{phiv} we have $\varphi_v\in \E_0^D(0,T)$, and we can use it as a test function in \eqref{dafermos2} to deduce
\begin{align}
-\int_0^T(\dot u(t),\dot v(t)) \de t+\int_0^T((\A &+\B) eu(t),e v(t)) \de t-\int_0^T\int_0^t\frac{1}{\beta}\e^{-\frac{t-\tau}{\beta}}(\B eu(\tau),e v(t)) \de\tau\de t\nonumber\\
&=\int_0^T( f (t),v(t)) \de t+\int_0^T( F(t),ev (t))\de t+(  u^1,v(0)).\label{mod2}
\end{align}
By taking the difference between \eqref{modM} and \eqref{mod2} we get $\langle u^1-\dot u(0),v(0)\rangle=0$ for every $v\in \E^D_T(0,T)$. Since for every $\vs\in V^D_0$ there exists a function $v\in \E^D_T(0,T)$ such that $v(0)=\vs$, we can obtain that $\langle u^1-\dot u(0),\vs\rangle=0$ for every $\vs\in V^D_0$, and so $u^1-\dot u(0)=0$ as element of $(V^D_0)'$. This proves the first part of the proposition.

Vice versa, let $u\in\V^D(0,T)$ be a weak solution in the sense of Definition \ref{newdef2}. Therefore, $u$ satisfies \eqref{first_weak} for every $v\in\D(0,T)$, and as we have already shown before, $u$ satisfies \eqref{modM}, with $u^1$ in place of $\dot u(0)$, for every function  $v\in \LT$. Let us consider $\varphi \in \E_0^D(0,T)$, then $v_{\varphi}(t)=(t-T)\dot \varphi(t)\in \LT$, and so it can be used as a test function in \eqref{modM}. By noticing that $\dot v_{\varphi}(t)=\dot\varphi(t)+(t-T)\ddot \varphi(t)$ and $v_{\varphi}(0)=-T\dot \varphi(0)$ we obtain the thesis.
\end{proof}

In view of the previous proposition, it will be enough to prove the existence of a solution to Dafermos' Equality \eqref{dafermos2}. In particular, we shall prove the existence of $t_0\in(0,T]$ and of a function $u\in\V^D(0,t_0)$ such that $u(0)=0$, and which satisfies Dafermos' Equality on the interval $[0,t_0]$. In order to do this, we use an abstract result due to Lions (see \cite[Chapter 3, Theorem 1.1 and Remark 1.2]{LL}). We first introduce the necessary setting. Let $X$ be a Hilbert space and $Y\subset X$ be a linear subspace, endowed with the scalar product $(\cdot,\cdot)_Y$ which makes it a pre-Hilbert space. Suppose that the inclusion of $Y$ in $X$ is a continuous map, i.e., there exists a positive constant $C$ such that
\begin{equation}\label{inclusion}
    \|u\|_X\leq C\|u\|_Y\quad \text{for every $u\in Y$}.
\end{equation}
Let us consider a bilinear form $B:X\times Y\rightarrow \R$ such that
\begin{alignat}{4}
&B(\cdot,\varphi):X\rightarrow \R &&\quad \text{is a linear continuous function on $X$ for every $\varphi\in Y$},\label{proj2}\\
&B(\varphi,\varphi)\geq \alpha \|\varphi\|^2_Y &&\quad \text{for every $\varphi\in Y$}, \text{ for some positive constant $\alpha$}.\label{proj3}
\end{alignat}

Now, we can state the aforementioned existence theorem.

\begin{theorem}[J.L. Lions]\label{Lions}
Suppose that hypotheses \eqref{inclusion}--\eqref{proj3} are satisfied, and let $L:Y\rightarrow\R$ be a linear continuous map. Then there exists $u\in X$ such that
\begin{equation*}
    B(u,\varphi)=L(\varphi) \qquad \text{for every $\varphi\in Y$.}
\end{equation*}
Moreover, the solution $u$ satisfies
\begin{equation}\label{bound_su_u}
    \|u\|_X\leq\frac{C}{\alpha}\sup\{|L(\varphi)|:\|\varphi\|_{Y}=1\}.
\end{equation}
\end{theorem}

After defining for every $a,b\in[0,T]$ with $a<b$ the space
\begin{equation*}
    \V_0^D(a,b):=\{u\in\V^D(a,b): u(a)=0\},
\end{equation*}
we can state the following proposition.
\begin{proposition}\label{step-one}
There exists $t_0\in (0,T]$ and a function $u\in \V_0^D(0,t_0)$ which satisfies Dafermos' Equality \eqref{dafermos2} on the interval $[0,t_0]$ for every $\varphi\in \E_0^D(0,t_0)$. Moreover, there exists a positive constant $C_0=C_0(t_0,\A)$ such that
\begin{equation}\label{stima_sol_0}
     \|u\|_{\V(0,t_0)}\leq C_0\left(\|f\|_{L^2(0,t_0;H)}+\|F\|_{H^1(0,t_0;H^d_s)}+\|u^1\|\right).
 \end{equation}

\end{proposition}
\begin{proof}
We fix $t_0\in (0,T]$ such that
\begin{equation}\label{tzero}
\begin{cases}
t_0<\frac{1}{2C_{\A }}&\text{if $\frac{1}{2C_{\A }}<T$}\\
t_0=T& \text{otherwise}.
\end{cases}
\end{equation}
For simplicity of notation, we denote the spaces $\V_0^D(0,t_0)$ and $\E_0^D(0,t_0)$ with the symbols $\V_{t_0}$ and $\E_{t_0}$, respectively. On the space $\V_{t_0}$ we take the usual scalar product, instead on the space $\E_{t_0}$ we consider the following one 
$$(\phi,\varphi)_{\E_{t_0}}:=\int_0^{t_0}[(\dot \phi(t),\dot \varphi(t))+(\phi(t),\varphi(t))_{V}]\de t+t_0(\dot \phi(0),\dot\varphi(0)) \qquad \text{for every $\phi,\varphi\in\E_{t_0},$}$$
and we denote by $\|\cdot\|_{\mathcal{E}_{t_0}}$ the norm associated.

Let us consider the bilinear form $B:\V_{t_0}\times \E_{t_0}\rightarrow \R $ defined by
\begin{equation*}
    B(u,\varphi):= \int_0^{t_0}(\dot u(t),\dot \varphi (t))+(t-t_0)\Big[(\dot u(t),\ddot \varphi(t))-\Big((\A +\B) eu(t)-\int_0^t\frac{1}{\beta}\e^{-\frac{t-\tau}{\beta}}\B eu(\tau)\de\tau,e \dot \varphi(t)\Big)\Big]\de t,
\end{equation*}
and the linear operator $L:\E_{t_0}\rightarrow \R$ represented by
\begin{align*}
    L(\varphi):=t_0(u^1,\dot \varphi&(0))+\int_0^{t_0} (t-t_0)(\dot{ F}(t),e\varphi(t))\de t+\int_0^{t_0} ( F(t),e\varphi(t))\de t-\int_0^{t_0} (t-t_0)( f (t),\dot\varphi(t))\de t.
\end{align*}
Notice that, from these definitions, Dafermos' Equality \eqref{dafermos2} on the interval $[0,t_0]$ can be rephrased as follows
\begin{equation*}
    B(u,\varphi)=L(\varphi) \qquad \text{for every $\varphi\in \E_{t_0}$}.
\end{equation*}
Now we are in the framework of Theorem \ref{Lions}, and we want to show that \eqref{proj2} and \eqref{proj3} are satisfied. Foremost, we prove the existence of a positive constant $\alpha$ such that
\begin{equation*}
    B(\varphi,\varphi)\geq \alpha\norm{\varphi}^2_{\E_{t_0}} \qquad \text{for every $\varphi\in \E_{t_0}$}.
\end{equation*}
By definition we have
\begin{equation}\label{Bco}
    B(\varphi,\varphi)= \int_0^{t_0}\norm{\dot \varphi(t)}^2+(t-t_0)\Big[(\dot \varphi(t),\ddot \varphi(t))-((\A +\B) e\varphi(t),e \dot \varphi(t))+\int_0^t\frac{1}{\beta}\e^{-\frac{t-\tau}{\beta}}(\B e\varphi(\tau),e \dot \varphi(t)) \de\tau\Big]\de t.
\end{equation}
Now we define
\begin{equation*}
    \psi(t):=\int_0^t\frac{1}{\beta}\e^{-\frac{t-\tau}{\beta}}e\varphi(\tau)\de \tau \quad \text{and consequently we have}\quad  \dot\psi(t)=\frac{1}{\beta}e\varphi(t)-\int_0^t\frac{1}{\beta^2}\e^{-\frac{t-\tau}{\beta}}e\varphi(\tau)\de \tau;
\end{equation*}
then \eqref{Bco} can be reworded as
\begin{equation}\label{Bco2}
    B(\varphi,\varphi)= \int_0^{t_0}\norm{\dot \varphi(t)}^2+(t-t_0)[(\dot \varphi(t),\ddot \varphi(t))-((\A +\B) e\varphi(t),e \dot \varphi(t))+(\B\psi(t),e\dot\varphi(t))]\de t.
\end{equation}
Thanks to the chain rule and to the symmetry property \eqref{CB1.5}, we can write
\begin{align*}
    &\frac{1}{2}\frac{\de }{\de t}\norm{\dot \varphi(t)}^2=(\dot \varphi(t),\ddot \varphi(t)), \qquad  \frac{1}{2}\frac{\de }{\de t}((\A +\B)  e \varphi(t),e\varphi(t))=((\A +\B)  e \varphi(t),e\dot \varphi(t)),\\
    &\frac{\de}{\de t}(\B \psi(t),e\varphi(t))=(\B \dot \psi(t),e\varphi(t))+(\B \psi(t),e\dot \varphi(t)).
\end{align*}
By substituting this information in \eqref{Bco2}, we get after some integration by parts
\begin{align}\label{BB}
    B(\varphi,\varphi)&=\int_0^{t_0}\norm{\dot \varphi(t)}^2\de t+\frac{1}{2}\int_0^{t_0}(t-t_0)\frac{\de}{\de t}\norm{\dot\varphi(t)}^2\de t-\frac{1}{2}\int_0^{t_0}(t-t_0)\frac{\de}{\de t}((\A +\B)  e\varphi(t),e\varphi(t))\de t\nonumber\\
    &\hspace{2.15cm}+\int_0^{t_0}(t-t_0)\frac{\de}{\de t}(\B\psi(t),e\varphi(t))\de t-\int_0^{t_0}(t-t_0)(\mathbb B \dot \psi(t),e \varphi(t)) \de t\nonumber\\
    &=\frac{t_0}{2}\|\dot \varphi(0)\|^2+\frac{1}{2}\int_0^{t_0}\|\dot \varphi(t)\|^2\de t+\frac{1}{2}\int_0^{t_0}((\A +\B)  e\varphi(t),e \varphi(t))\de t\nonumber\\
    &\hspace{2.15cm}-\int_0^{t_0}(\B \psi(t),e\varphi(t))\de t-\int_0^{t_0}(t-t_0)(\mathbb B \dot \psi(t),e \varphi(t)) \de t\nonumber\\
    &=\frac{t_0}{2}\|\dot \varphi(0)\|^2+\frac{1}{2}\int_0^{t_0}\|\dot \varphi(t)\|^2\de t+\frac{1}{2}\int_0^{t_0}((\A +\B)  e\varphi(t),e \varphi(t))\de t\nonumber\\
    &\hspace{2.15cm}-\int_0^{t_0}(t-t_0)(\beta\B\dot\psi(t),\dot \psi(t))\de t-\int_0^{t_0}(t-t_0)(\B\dot \psi(t),\psi(t))-\int_0^{t_0}(\B\psi(t),e\varphi(t))\de t\nonumber\\
    &=\frac{t_0}{2}\|\dot \varphi(0)\|^2+\frac{1}{2}\int_0^{t_0}\|\dot \varphi(t)\|^2\de t+\frac{1}{2}\int_0^{t_0}(\A   e\varphi(t),e \varphi(t))\de t\nonumber\\
    &\hspace{2.15cm}+\frac{1}{2}\int_0^{t_0}(\B  (e\varphi(t)-\psi(t)),e\varphi(t)-\psi(t))\de t+\int_0^{t_0}(t_0-t)(\beta\B\dot\psi(t),\dot \psi(t))\de t.
\end{align}
From the coerciveness in \eqref{CB2M} and the definition of the $V$-norm, we have
\begin{equation}\label{coercM}
    (\A   e\varphi(t),e\varphi(t))\geq C_{\A }\|\varphi(t)\|^2_{V}-C_{\A }\| \varphi(t)\|^2 \qquad \text{for every $t\in [0,T]$}.
\end{equation}
Moreover, since 
$$\varphi(t)=\varphi(0)+\int_0^t\dot\varphi(\tau)\de \tau=\int_0^t\dot\varphi(\tau)\de \tau,$$
inequality \eqref{coercM} implies
\begin{equation}\label{BB*}
    \frac{1}{2}\int_0^{t_0}(\A  e\varphi(t),e\varphi(t))\de t\geq \frac{C_{\A }}{2}\int_0^{t_0}\|\varphi(t)\|^2_{V}\de t-\frac{C_{\A } t_0}{2}\int_0^{t_0}\|\dot \varphi(t)\|^2\de t.
\end{equation}
By \eqref{BB}, \eqref{BB*}, and in view of the choice done in \eqref{tzero}, we can deduce
\begin{equation*}
    B(\varphi,\varphi)\geq \frac{t_0}{2}\|\dot \varphi(0)\|^2+\frac{1-C_{\A } t_0}{2}\int_0^{t_0}\|\dot \varphi(t)\|^2\de t+\frac{C_{\A }}{2}\int_0^{t_0}\|\varphi(t)\|^2_{V} \de t\geq \frac{1}{4}\min\{1,C_{\A }\}\|\varphi\|^2_{\E_{t_0}}, 
\end{equation*}
which corresponds to the hypothesis \eqref{proj3}, with 
\begin{equation}\label{alpha}
\alpha=\frac{1}{4}\min\{1,C_{\A }\}.    
\end{equation}

We now show the validity of assumption \eqref{proj2}. We have to prove that for every $\varphi\in \E_{t_0}$ the functional $B(\cdot,\varphi)$ is continuous on $\V_{t_0}$, and that $L:\E_{t_0}\rightarrow \R$ is a linear continuous operator on the space $\E_{t_0}$. To this aim, we fix $\varphi\in \E_{t_0}$ and we consider $\{u_k\}_k\subset \V_{t_0}$ such that 
$$u_k\xrightarrow[k\to\infty]{\V_{t_0}} u .$$
Therefore
$$U_k:=u_k-u\xrightarrow[k\to\infty]{L^2(0,t_0;V)}0 \quad \text{and} \quad \dot U_k:=\dot u_k-\dot u\xrightarrow[k\to\infty]{L^2(0,t_0;H)}0. $$
By using Cauchy-Schwarz's inequality we get
\begin{align}
    |B(U_k,\varphi)|&\leq \int_0^{t_0}|(\dot U_k(t),\dot \varphi(t))| \de t+t_0\int_0^{t_0}|(\dot U_k(t),\ddot \varphi(t))|\de t+t_0\int_0^{t_0}|((\A +\B) e U_k(t),e \dot \varphi(t))|\de t\nonumber\\
    &\hspace{3.6cm}+t_0\int_0^{t_0}\int_0^t\frac{1}{\beta}\e^{-\frac{t-\tau}{\beta}}|(\mathbb B e U_k(\tau),e \dot \varphi(t))| \de\tau\de t\nonumber\\
    &\leq \|\dot U_k\|_{L^2(0,t_0;H)}\|\dot \varphi\|_{L^2(0,t_0;H)}+t_0\|\dot U_k\|_{L^2(0,t_0;H)}\|\ddot \varphi\|_{L^2(0,t_0;H)}\nonumber\\
    &\qquad\qquad+t_0\|\A +\B\|_{{\infty}}\|U_k\|_{L^2(0,t_0;V)}\|\dot \varphi\|_{L^2(0,t_0;V)}+\frac{t_0}{\beta}\|\B\|_{{\infty}}\int_0^{t_0}\int_0^t |(e U_k(\tau),e\dot \varphi(t))|\de \tau\de t\label{star}.
\end{align}
Notice that
\begin{align*}
    \int_0^{t_0}\int_0^t |(e U_k(\tau),e\dot \varphi(t))|\de \tau\de t&\leq \|\dot \varphi\|_{L^2(0,t_0;V)}\Big(\int_0^{t_0}\Big(\int_0^t \| U_k(\tau)\|_V\de \tau\Big)^2\de t\Big)^{\frac{1}{2}}\leq t_0 \|\dot \varphi\|_{L^2(0,t_0;V)} \|U_k\|_{L^2(0,t_0;V)},
\end{align*}
whence, by considering \eqref{star}, we can say that there exist two positive constants $C_1=C_1(\varphi,t_0)$ and $C_2=C_2(\A,\B,t_0,\beta,\varphi)$ such that
\begin{equation*}
     |B(U_k,\varphi)|\leq C_1\|\dot U_k\|_{L^2(0,t_0;H)}+C_2\|U_k\|_{L^2(0,t_0,V)}\xrightarrow[k\to\infty]{}0.
\end{equation*}
Now it remains to show that $L$ is a continuous operator on $\E_{t_0}$, and since it is linear it is enough to show its boundedness. Let $\varphi\in \E_{t_0}$, then
\begin{align}\label{b-1}
    \hspace{-0.2cm}|L(\varphi)|&\leq \Big|\int_0^{t_0}\left[(t-t_0)( f (t),\dot\varphi(t))-(t-t_0)(\dot{ F}(t),e\varphi(t))-( F(t),e\varphi(t))\right]\de t\Big|+t_0\|u^1\|\|\dot \varphi(0)\|.
\end{align}
In particular there exists a positive constant $C=C(f, F,t_0)$ such that 
\begin{align}
    \int_0^{t_0}|(t-t_0)( f (t),\dot\varphi(t))&-( F(t),e\varphi(t))-(t-t_0)(\dot{ F}(t),e\varphi(t))|\de t\nonumber\\
    & \leq t_0\| f \|_{L^2(0,t_0;H)}\|\dot \varphi\|_{L^2(0,t_0;H)}+\Big(\int_0^{t_0}\|(t-t_0)\dot{ F}(t)+ F(t)\|^2\de t\Big)^{\frac{1}{2}}\| \varphi\|_{L^2(0,T;V)}\nonumber\\
    & \leq t_0\| f \|_{L^2(0,t_0;H)}\|\varphi\|_{\E_{t_0}} +2^{\frac{1}{2}}\max\{t_0,1\}\| F\|_{H^1(0,t_0;H^d_s)}\|\varphi\|_{\E_{t_0}}\leq C \|\varphi\|_{\E_{t_0}}.
\end{align}
Moreover, we have
\begin{equation}\label{b-3}
    t_0\|u^1\|\|\dot \varphi(0)\|\leq t_0 \|u^1\| t_0^{-\frac{1}{2}}\| \varphi\|_{\E_{t_0}}=t_0^{\frac{1}{2}} \|u^1\| \| \varphi\|_{\E_{t_0}}.
\end{equation}
By applying Theorem \ref{Lions} with $X=\V_{t_0}$ and $Y=\E_{t_0}$, we have the existence of a solution to \eqref{dafermos2} on the interval $[0, t_0]$.

Furthermore, we can use \eqref{bound_su_u} and \eqref{alpha}, and by means of \eqref{b-1}--\eqref{b-3} we obtain \eqref{stima_sol_0} with
\begin{equation*}
    C_0:=\frac{\max\{2^{\frac{1}{2}}\max\{t_0,1\},t_0^{\frac{1}{2}}\}}{\frac{1}{4}\min\{1,C_{\A}\}}.
\end{equation*} 
\end{proof}

\begin{remark}\label{remi2}
At this point, from Remark \ref{remi} and Propositions \ref{daf} and \ref{step-one}, we can find a weak solution to the viscoelastic dynamic system \eqref{classic_model_inf3}--\eqref{sys_inf3} on the interval $[0,t_0]$.
\end{remark}

Now we want to show that it is possible to find a weak solution on the whole interval $[0,T]$. Let $b,c\in[t_0,T)$ be two real numbers such that $b<c$, then we can state the following lemma. 

\begin{lemma}\label{congiun}
Let $u\in \V^D(0,b)$ be a function which satisfies \eqref{first_weak} on the interval $[0,b]$, then the following equality holds
\begin{align}\label{ex1}
    \langle \dot u(b),\psi(b)\rangle-\int_{0}^{b}(\dot u(t),\dot \psi(t))\de t&+\int_{0}^{b}((\A +\B)  eu(t),e\psi(t))\de t-\int_0^b\int_0^t\frac{1}{\beta}\e^{-\frac{t-\tau}{\beta}}(\B eu(\tau),e\psi(t))\de \tau\de t\nonumber\\
    &=\int_{0}^{b}(f(t),\psi(t))\de t+\int_{0}^{b}(F(t),e\psi(t))\de t,
\end{align}
 for every $\psi \in \V^D(0,b)$ such that $\psi(0)=0$.
 
Moreover, if $u\in\V^D(b, c )$ is a function which satisfies \eqref{first_weak} on the interval $[b, c]$, then the following equality holds
\begin{align}\label{ex2}
    -\langle \dot u(b),\Psi(b)\rangle-\int_b^c(\dot u(t),\dot \Psi(t))\de t&+\int_b^c((\A +\B)  eu(t),e\Psi(t))\de t-\int_b^c\int_{b}^t\frac{1}{\beta}\e^{-\frac{t-\tau}{\beta}}(\B eu(\tau),e\Psi(t))\de \tau\de t\nonumber\\
    &=\int_b^c(f(t),\Psi(t))\de t+\int_b^c(F(t),e\Psi(t))\de t,
\end{align}
for every $\Psi \in \V^D(b, c )$ such that $\Psi( c )=0$.
\end{lemma}
\begin{proof}
We begin by proving \eqref{ex1}. We consider $\psi \in \V^D(0,b)$ such that $\psi(0)=0$, and we define for $\varepsilon\in(0,b)$ the function
$$\psi_{\varepsilon}(t)=\begin{cases}
\psi(t) & t\in [0,b-\varepsilon]\\
\frac{b-t}{\varepsilon}\psi(t)& t\in [b-\varepsilon, b].
\end{cases}$$
Since $\psi_{\varepsilon}\in \V^D(0,b)$ and $\psi_{\varepsilon}(0)=\psi_{\varepsilon}(b)=0$, we can use it as a test function in \eqref{first_weak} to obtain $I_{\varepsilon}+J_{\varepsilon}=K_{\varepsilon}$, where
\begin{align*}
    &I_{\varepsilon}:=-\int_0^{b-\varepsilon}(\dot u(t),\dot \psi(t)) \de t+\dashint_{b-\varepsilon}^{b}(\dot u(t),\psi(t))\de t+\int_0^{b-\varepsilon}\Big((\A +\B) eu(t)-\hspace{-1pt}\int_0^{b-\varepsilon}\hspace{-4pt}\int_0^t\frac{1}{\beta}\e^{-\frac{t-\tau}{\beta}}\B eu(\tau)\de\tau,e \psi(t)\Big)\de t,\\
    &J_{\varepsilon}:=-\dashint_{b-\varepsilon}^{b}(b-t)(\dot u(t),\dot \psi(t))\de t+\dashint_{b-\varepsilon}^{b}(b-t)\Big((\A +\B)  eu(t)-\int_0^t\frac{1}{\beta}\e^{-\frac{t-\tau}{\beta}}\B eu(\tau)\de \tau,e\psi(t)\Big)\de t,\\
    &K_{\varepsilon}:=\int_0^{b-\varepsilon}(f(t),\psi (t)) \de t+\dashint_{b-\varepsilon}^{b}(b-t)( f(t),\psi(t))\de t+\int_0^{b-\varepsilon}(F(t),e\psi (t)) \de t+\dashint_{b-\varepsilon}^{b}(b-t)(F(t),e\psi(t))\de t.
\end{align*}
Thanks to the absolute continuity of Lebesgue's integral and to Remark \ref{udotcon} we get
\begin{align*}
    &I_{\varepsilon}\xrightarrow[\varepsilon\to 0^+]{}-\int_0^{b}(\dot u(t),\dot \psi(t)) \de t+\int_0^{b}\Big((\A +\B) eu(t)-\int_0^t\frac{1}{\beta}\e^{-\frac{t-\tau}{\beta}}\B eu(\tau)\de \tau,e \psi(t)\Big)\de t+\langle \dot u(b),\psi(b)\rangle ,\\
    &J_{\varepsilon}\xrightarrow[\varepsilon\to 0^+]{}0, \qquad K_{\varepsilon}\xrightarrow[\varepsilon\to 0^+]{}\int_0^{b}( f(t),\psi (t)) \de t+\int_0^{b}(F(t),e\psi (t)) \de t,
\end{align*}
which concludes the proof of \eqref{ex1}.

To prove \eqref{ex2}, it is enough to consider for $\varepsilon\in (0,c-b)$ the function
$$\Psi_{\varepsilon}(t)=\begin{cases}
\frac{t-b}{\varepsilon}\Psi(t) & t\in [b,b+\varepsilon]\\
\Psi(t)& t\in [b+\varepsilon,  c ]
\end{cases}$$
where $\Psi\in\V^D(b, c )$ such that $\Psi( c )=0$, and to repeat similar argument before performed.
\end{proof}

Taking into account the previous lemma we can state and prove the following proposition.

\begin{proposition}\label{conj}
 Let $\tul\in \V^D(0,b)$ be a weak solution to the viscoelastic dynamic system \eqref{classic_model_inf3}--\eqref{sys_inf3} in the sense of Definition \ref{newdef2} on the interval $[0,b]$ which satisfies for some positive constants $\tilde C$ the following estimate
\begin{equation}\label{stima-nuova}
     \|\tul\|_{\V(0,b)}\leq \tilde C\left(\|f\|_{L^2(0,b;H)}+\|F\|_{H^1(0,b;H^d_s)}+\|u^1\|\right).
 \end{equation}
 Then, for every $l\geq 1$ there exists $c\in(b, b+\frac{t_0}{l}]$ such that we can extend $\tul$ to a function $u\in \V^D(0,c)$ which is a weak solution on the interval $[0,c]$. Moreover $u$ satisfies for some positive constants $C$ the following estimate
 \begin{equation}\label{stima-nuova2}
     \|u\|_{\V(0,c)}\leq C\left(\|f\|_{L^2(0,c;H)}+\|F\|_{H^1(0,c;H^d_s)}+\|u^1\|\right).
 \end{equation}
\end{proposition}

\begin{proof}
We divide the proof into two steps. In the first one, we show how to extend the solution. After this, in the second step, we prove \eqref{stima-nuova2}. We firstly choose $\hat b\in (b-\frac{t_0}{2l},b)$ in such a way that 
\begin{itemize}
    \item $\tul(\hat b)\in V$ and 
    \begin{equation}\label{media_1}
        \|\tul(\hat b)\|^2_V\leq \dashint_{b-\frac{t_0}{2l}}^{b}\|\tul(t)\|^2_V\de t;
    \end{equation}
    \item $\hat b$ is a Lebesgue's point for $\dtul$, that is 
\begin{equation}\label{media1/2}
    \lim_{\varepsilon\to 0^+}\dashint_{\hat b}^{\hat b+\varepsilon}\|\dtul(t)-\dtul (\hat b)\|\de t=0,
\end{equation}
and $\dtul(\hat b )\in H$ satisfies
\begin{equation}\label{media_2}
        \|\dtul(\hat b)\|^2\leq \dashint_{b-\frac{t_0}{2l}}^{b}\|\dtul(t)\|^2\de t.
    \end{equation}
\end{itemize}
Notice that \eqref{media_1}--\eqref{media_2} are possible because $\tul\in\V(0,b)$.

\textit{Step 1}. Since $\tul$ is a weak solution on the interval $[0,b]$, then 
\begin{align*}
     -\int_0^{b}(\dtul(t),\dot v(t)) \de t+\int_0^{b}((\A +\B) e\tul(t),ev(t))\de t&- \int_0^{b}\int_0^t\frac{1}{\beta}\e^{-\frac{t-\tau}{\beta}}(\B e\tul(\tau),e v(t))\de \tau \de t\nonumber\\
     &=\int_0^{b}( f(t),v (t)) \de t+\int_0^{b}( F(t),ev (t)) \de t,
\end{align*}
for every $v\in\V^D(0,b)$ such that $v(0)=v(b)=0$, and moreover $\tul$ satisfies
\begin{equation}\label{data-in}
    \lim_{t\to 0^+}\|\tul(t)\|=0\qquad\text{and}\qquad\lim_{t\to 0^+}\|\dtul(t)-u^1\|_{(V^D_0)'}=0.
\end{equation}
We define the function $G\in H^1(\hat b,\hat b+\frac{t_0}{l};H^d_s)$ in the following way
\begin{equation*}
    G(t):=F(t)+\int_0^{\hat b}\frac{1}{\beta}\e^{-\frac{t-\tau}{\beta}}\B e\tul(\tau)\de\tau.
\end{equation*}
Since $\frac{t_0}{l}\leq t_0$, $\tul(\hat b)\in V$, and $\dtul(\hat b)\in H$, we can apply Remark \ref{remi}, Propositions \ref{daf} and \ref{step-one} on the interval $[\hat b, \hat b+\frac{t_0}{l}]$, to find a function $\bul\in \V^D(\hat b, \hat b+\frac{t_0}{l} )$ which satisfies, for every $v\in \V^D(\hat b, \hat b+\frac{t_0}{l} ) $ such that $v(\hat b)=v( \hat b+\frac{t_0}{l} )=0$, the following equality 
\begin{align*}
     -\int_{\hat b}^{ \hat b+\frac{t_0}{l} }(\dbul(t),\dot v(t)) \de t+\int_{\hat b}^{ \hat b+\frac{t_0}{l} }((\A +\B) e\bul(t)&,ev(t))\de t-\int_{\hat b}^{ \hat b+\frac{t_0}{l} }\int_{\hat b}^t\frac{1}{\beta}\e^{-\frac{t-\tau}{\beta}}(\B e\bul(\tau),e v(t)) \de \tau\de t\nonumber\\
     &=\int_{\hat b}^{ \hat b+\frac{t_0}{l} }(f(t),v (t)) \de t+\int_{\hat b}^{ \hat b+\frac{t_0}{l} }(G(t),ev(t))\de t,
\end{align*}
and also the following limits
\begin{equation}\label{indat*}
    \lim_{t\to \hat b^+} \|\bul(t)-\tul(\hat b)\|=0,\quad \lim_{t\to \hat b^+}\|\dbul(t)-\dtul(\hat b)\|_{(V_0^D)'}=0.
 \end{equation}
Notice that the initial data $\tul(\hat b)$ and $\dtul(\hat b)$ are well defined because $\tul\in C^0([0,b];H)$ and $\dtul\in C^0([0,b];(V_0^D)')$. 

Now we define the function
\begin{align}\label{def-u}
u(t):=\begin{cases}
\tul(t)&t\in[0,\hat b]\\
\bul(t)&t\in[\hat b, \hat b+\frac{t_0}{l}],
\end{cases}
\end{align}
 and we claim that it is a weak solution on the interval $[0,\hat b+\frac{t_0}{l}]$. Notice that, since $\hat{b}\geq b-\frac{t_0}{2l}$ then $\hat{b}+\frac{t_0}{l}> b$. To prove this, let us fix $\psi\in \D(0,\hat b+\frac{t_0}{l})$. Clearly $\psi\in \V^D(0,\hat b)$ and $\psi( 0 )=0$, and since $\tul$ is a weak solution on $[0,\hat{b}]$, we can use \eqref{ex1} of Lemma \ref{congiun} to get
\begin{align}\label{fin1*}
    (\dtul(\hat b),\psi(\hat b))-\int_{0}^{\hat b}(\dot u(t),\dot \psi(t))\de t&+\int_{0}^{\hat b}((\A +\B)  eu(t),e\psi(t))\de t-\int_0^{\hat b}\int_0^t\frac{1}{\beta}\e^{-\frac{t-\tau}{\beta}}(\B eu(\tau)\de \tau,e\psi(t))\de t\nonumber\\
    &=\int_{0}^{\hat b}(f(t),\psi(t))\de t+\int_{0}^{\hat b}(F(t),e\psi(t))\de t.
\end{align}
Moreover, $\psi\in \V^D(\hat b, \hat b+\frac{t_0}{l} )$ and $\psi( \hat b+\frac{t_0}{l} )=0$, and since $\bar u$ is a weak solution on $[\hat{b},\hat{b}+\frac{t_0}{l}]$, by \eqref{ex2} of Lemma \ref{congiun} we obtain
\begin{align*}
    -(\dbul(\hat b),\psi(\hat b))&-\int_{\hat b}^{ \hat b+\frac{t_0}{l} }(\dot u(t),\dot \psi(t))\de t+\int_{\hat b}^{ \hat b+\frac{t_0}{l} }((\A +\B)  eu(t),e\psi(t))\de t\nonumber\\
    &-\int_{\hat b}^{\hat b+\frac{t_0}{l}}\int_{\hat b}^t\frac{1}{\beta}\e^{-\frac{t-\tau}{\beta}}(\B eu(\tau)\de \tau,e\psi(t))\de t=\int_{\hat b}^{ \hat b+\frac{t_0}{l} }(f(t),\psi(t))\de t+\int_{\hat b}^{ \hat b+\frac{t_0}{l} }(G(t),e\psi(t))\de t,
\end{align*}
that is
\begin{align}\label{fin2*}
    -(\dbul(\hat b),\psi(\hat b))&-\int_{\hat b}^{ \hat b+\frac{t_0}{l} }(\dot u(t),\dot \psi(t))\de t+\int_{\hat b}^{ \hat b+\frac{t_0}{l} }((\A +\B)  eu(t),e\psi(t))\de t\nonumber\\
    &-\int_{\hat b}^{\hat b+\frac{t_0}{l}}\int_{0}^t\frac{1}{\beta}\e^{-\frac{t-\tau}{\beta}}(\B eu(\tau)\de \tau,e\psi(t))\de t=\int_{\hat b}^{ \hat b+\frac{t_0}{l} }(f(t),\psi(t))\de t+\int_{\hat b}^{ \hat b+\frac{t_0}{l} }(F(t),e\psi(t))\de t.
\end{align}
From \eqref{media1/2} and \eqref{indat*}, by summing \eqref{fin1*} and \eqref{fin2*}, we obtain the following equality
\begin{align}\label{ww}
     -\int_0^{ \hat b+\frac{t_0}{l} }(\dot u(t),\dot \psi(t)) \de t&+\int_0^{ \hat b+\frac{t_0}{l} }((\A +\B) eu(t),e\psi(t))\de t-\int_0^{ \hat b+\frac{t_0}{l} }\int_0^t\frac{1}{\beta}\e^{-\frac{t-\tau}{\beta}}(\B eu(\tau),e \psi(t)) \de\tau\de t\nonumber\\
    &=\int_0^{ \hat b+\frac{t_0}{l} }(f(t),\psi (t)) \de t+\int_0^{ \hat b+\frac{t_0}{l} }(F(t),e\psi (t)) \de t.
\end{align}
By setting $c:=\hat b+\frac{t_0}{l}$ we have that the function $u$ defined in \eqref{def-u} is a weak solution to the viscoelastic dynamic system \eqref{classic_model_inf3}--\eqref{sys_inf3} in the sense of Definition \ref{newdef2} on the interval $[0,c]$, since it satisfies \eqref{data-in} and \eqref{ww}.

\textit{Step 2}. Now, we want to prove \eqref{stima-nuova2}. We can write 
 \begin{equation}\label{daf-stima-0}
      \|u\|^2_{\V(0,c)}=\|\tul\|^2_{\V(0,\hat b)}+\|\bul\|^2_{\V(\hat b,c)}\leq \|\tul\|^2_{\V(0, b)}+\|\bul\|^2_{\V(\hat b,c)}.
 \end{equation}
Notice that $\bul-\tul(\hat b)\in\V_0^D(\hat b,c)$ is a function which satisfies Dafermos' Equality \eqref{dafermos2} on the interval $[\hat b,c]$ with the right-hand side equal to
\begin{equation*}
    t_0(\dot{\tul}(\hat b),\dot \varphi(0))-\int_{\hat b}^{c}(t-t_0)\Big[( f (t),\dot \varphi(t))+( G(t)-\A e\tul(\hat b)-\e^{-\frac{t-\hat b}{\beta}}\B e\tul(\hat b), e\dot\varphi(t))\Big] \de t\quad\text{for every $\varphi\in \E_0^D(\hat b,c)$.}
\end{equation*}
Therefore, by following the estimates in \eqref{b-1}--\eqref{b-3}, we can apply \eqref{bound_su_u} of Theorem \ref{Lions}, with $X=\V^D(\hat b,c)$ and $Y=\E_0^D(\hat b,c)$, to obtain the existence of a positive constant $K=K(t_0,\A)$ such that
\begin{align}\label{daf2}
    \|\bul-\tul(\hat b)&\|_{\V(\hat b,c)}\leq K\big[\|f\|_{L^2(\hat b,c;H)}+\| G-\A e\tul(\hat b)-\e^{-\frac{\cdot-\hat b}{\beta}}\B e\tul(\hat b)\|_{H^1(\hat b,c;H^d_s)}+\|\dot{\tul}(\hat b)\|\big].
\end{align}
Now notice that 
\begin{align}
    \|G\|_{H^1(\hat b,c;H^d_s)}&\leq \|F\|_{H^1(\hat b,c;H^d_s)}+\Big(\frac{\beta}{2}\Big)^{\frac{1}{2}}\Big(1+\frac{1}{\beta}\Big)\|\B\|_{\infty}\Big(\int_0^{\hat b}\frac{1}{\beta^2}\e^{-\frac{2(\hat b-\tau)}{\beta}}\de\tau\Big)^{\frac{1}{2}}\|\tul\|_{L^2(0,\hat b;V)}\nonumber\\
    &\leq\|F\|_{H^1(\hat b,c;H^d_s)}+\frac{1}{2}\Big(1+\frac{1}{\beta}\Big)\|\B\|_{\infty}\|\tul\|_{\V(0,\hat b)},
\end{align}
and
\begin{align}\label{ult_k}
    \|\A e\tul(\hat b)+\e^{-\frac{\cdot-\hat b}{\beta}}\B e\tul(\hat b)\|_{H^1(\hat b,c;H^d_s)}&\leq \Big[\Big(\frac{t_0}{l}\Big)^{\frac{1}{2}}\|\A\|_{\infty}+\|\B\|_{\infty}\|\e^{-\frac{\cdot-\hat b}{\beta}}\|_{H^1(\hat b,c)}\Big]\|\tul(\hat b)\|_V\nonumber\\
    &\leq \Big[\Big(\frac{t_0}{l}\Big)^{\frac{1}{2}}\|\A\|_{\infty}+\Big(\frac{\beta}{2}\Big)^{\frac{1}{2}}\Big(1+\frac{1}{\beta}\Big)\|\B\|_{\infty}\Big]\|\tul(\hat b)\|_V.
\end{align}
Taking into account the information provided by \eqref{media_1}--\eqref{media_2}, we can use estimates \eqref{daf2}--\eqref{ult_k} to deduce the existence of a positive constant $\bar{C}=\bar{C}(t_0,l,\A,\B,\beta)$ such that
\begin{align}\label{ultimo_est}
      \|\bul\|_{\V(\hat b,c)}&\leq  \bar{C}\Big(\|f\|_{L^2(\hat b,c;H)}+\|F\|_{H^1(\hat b,c;H^d_s)}+\|\tul\|_{\V(0,b)}\Big).
 \end{align}
By \eqref{stima-nuova}, \eqref{daf-stima-0}, and \eqref{ultimo_est} we obtain the final estimate \eqref{stima-nuova2}.
\end{proof}

Now we are in position to prove the main theorem of this section.

\begin{proof}[Proof of Theorem~\ref{main1}]
Let us consider $u_0\in \V^D(0,t_0)$ a weak solution to the viscoelastic dynamic system \eqref{classic_model_inf3}--\eqref{sys_inf3} in the sense of Definition \ref{newdef2} on the interval $[0,t_0]$, whose existence is guaranteed by Remark \ref{remi2}. Moreover, $u_0$ satisfies \eqref{stima_sol_0}. By applying a finite number of times Proposition \ref{conj} with $l=1$ we can extend $u_0$ to $\tilde{u}\in \V^D(0,b)$ which is a weak solution on the interval $[0,b]$, where $T-b<t_0$. Now we select $\hat b\in (T-t_0,b)$ in such a way \eqref{media_1}--\eqref{media_2} are satisfied on the interval $[T-t_0,b]$. By choosing $l=\frac{t_0}{T-\hat{b}}\geq 1$, since $\hat b+\frac{t_0}{l}=T$, taking into account Proposition \ref{conj} we can extend $\tul$ to a function $u\in \V^D(0,T)$ which is a weak solution to the viscoelastic dynamic system \eqref{classic_model_inf3}--\eqref{sys_inf3} on the interval $[0,T]$. Moreover $u$ satisfies \eqref{stima-nuova2} on $[0,T]$. Finally, by applying Remark \ref{remi} we get the thesis.
\end{proof}

\section{Existence: A coupled system equivalent to the viscoelastic dynamic system}\label{sub2}
In this section, we illustrate a second method to find solutions to the viscoelastic dynamic system \eqref{classic_model_inf3}--\eqref{sys_inf3} according to Definition \ref{newdef2}. This method is based on a minimizing movement approach deriving from the theory of gradient flows, and it is a classical tool used to prove the existence of solutions in the context of fractures, see, e.g., \cite{sc}, \cite{DM-Lar}, \cite{DMT2}. By means of this method, we are also able to provide an energy-dissipation inequality satisfied by the solution, and consequently, thanks to this inequality, we prove that such a solution satisfies the initial conditions \eqref{sys_inf3} in a stronger sense than the one stated in \eqref{ini*}.

To this aim, let us define the following \textit{coupled} system
\begin{equation}\label{cs}
\begin{cases}
\ddot u(t)-\div(\A eu(t))-\div(\B(eu(t)-w(t)))=f(t)-\div (F(t)-\e^{-\frac{t}{\beta}}\B w^0) &\quad\text{in $\Omega\setminus\Gamma_t$, $t\in(0,T)$},\\
 \beta \dot w(t)+w(t)=eu(t) 
\end{cases}
\end{equation}
with the following boundary and initial conditions
\begin{alignat}{4}
&u(t)=z(t) && \qquad \text{on $\partial_D\Omega$}, && \quad t\in(0,T),\label{csb1}\\
&[\A eu(t)+\B (eu(t)-w(t))] \nu=(F(t)-\e^{-\frac{t}{\beta}}\B w^0)\nu && \qquad \text{on } \partial_N\Omega, &&\quad t\in(0,T),\label{csb2}\\
&[\A eu(t)+\B (eu(t)-w(t))] \nu=(F(t)-\e^{-\frac{t}{\beta}}\B w^0)\nu && \qquad \text{on }\Gamma_t,  &&\quad t\in(0,T),\label{csb3}\\
& u(0)=u^0, \quad w(0)=w^0,\quad  \dot u(0)=u^1,\label{csi} 
\end{alignat}
where $w^0\in H^d_s$. Also in this case, the strong formulation of the coupled system \eqref{cs}--\eqref{csi} is only formal. By setting
\begin{equation*}
\V:=\V(0,T),\quad \V^D:=\V^D(0,T), \quad \D:=\D(0,T), 
\end{equation*}
we give the following definition.

\begin{definition}\label{defsym}
We say that $(u,w)\in\V\times  H^1(0,T;H^d_s)$ is a {\it weak solution} to the coupled system \eqref{cs}--\eqref{csi} if the following conditions hold:
\begin{itemize} 
    \item $u-z\in \V^D$ and 
    \begin{align}\label{wweak}
     -\int_0^T(\dot u(t),\dot \varphi(t)) \de t&+\int_0^T(\A  eu(t),e \varphi(t)) \de t+\int_0^T(\B( eu(t)-w(t)),e \varphi(t)) \de t\nonumber\\
     &=\int_0^T(f(t),\varphi (t)) \de t+\int_0^T(F(t),e\varphi (t)) \de t-\int_0^T\e^{-\frac{t}{\beta}}(\B w^0,e\varphi (t)) \de t,
\end{align}
for every $\varphi \in \D$;
\item for a.e. $t\in (0,T)$ 
\begin{equation}\label{odes}
\begin{cases}
\beta\dot w(t)+w(t)=eu(t) \\
w(0)=w^0
\end{cases}
\end{equation}
where the equalities are to be understood in the sense of the Hilbert space $H^d_s$;
\item the initial conditions \eqref{ini*} are satisfied.
\end{itemize}
\end{definition}

The following result proves that the new problem is equivalent to the first one.
\begin{theorem}\label{equivalence}
The viscoelastic dynamic system \eqref{classic_model_inf3}--\eqref{sys_inf3} is equivalent to the coupled system \eqref{cs}--\eqref{csi}.
\end{theorem}
\begin{proof}
Let us consider a weak solution $(u,w)\in \V\times H^1(0,T;H^d_s)$ to the coupled system \eqref{cs}--\eqref{csi} according to Definition \ref{defsym}. In view of the theory of ordinary differential equations valued in Hilbert spaces, by \eqref{odes} we can write 
\begin{equation}\label{sol-ode}
    w(t)=w^0\e^{-\frac{t}{\beta}}+\int_0^t\frac{1}{\beta}\e^{-\frac{t-\tau}{\beta}} eu(\tau)\de \tau\qquad \text{for every $t\in [0,T]$}.
\end{equation}
Moreover, by definition $u-z\in \V^D$ and \eqref{wweak} holds for every $\varphi \in \D$. By substituting \eqref{sol-ode} in \eqref{wweak} we obtain
\begin{align*}
     -\int_0^T(\dot u(t),\dot \varphi(t)) \de t&+\int_0^T\Big((\A +\B) eu(t)-\int_0^t\frac{1}{\beta}\e^{-\frac{t-\tau}{\beta}}\B e u(\tau)\de \tau,e \varphi(t)\Big)\de t -\int_0^T \e^{-\frac{t}{\beta}}(\B w^0,e\varphi(t))\de t\nonumber   \\
     &=\int_0^T(f(t),\varphi (t)) \de t+\int_0^T(F(t),e\varphi (t)) \de t-\int_0^T \e^{-\frac{t}{\beta}}(\B w^0,e\varphi(t))\de t.
\end{align*}
Therefore, since, again by definition, \eqref{ini*} holds, $u$ is a weak solution to the viscoelastic dynamic system \eqref{classic_model_inf3}--\eqref{sys_inf3} in the sense of Definition \ref{newdef2}.

Vice versa, if we consider a solution $u\in \V$ to the viscoelastic dynamic system \eqref{classic_model_inf3}--\eqref{sys_inf3}, then $u-z\in \V^D$ and 
\begin{align}\label{mid}
-\int_0^T(\dot u(t),\dot \varphi (t)) \de t+\int_0^T((\A +\B) eu(t),e \varphi(t)) \de t&-\int_0^T\int_0^t\frac{1}{\beta}\e^{-\frac{t-\tau}{\beta}}(\B eu(\tau),e \varphi(t)) \de \tau\de t\nonumber\\
&=\int_0^T(f(t),\varphi(t)) \de t+\int_0^T(F(t),e\varphi (t)) \de t,
\end{align}
for every $\varphi\in\D$. Let $w^0\in H^d_s$ and let $w$ be the function defined in \eqref{sol-ode}. It is easy to see that $w\in H^1(0,T;H^d_s)$ and by summing to both hand sides of \eqref{mid} the term
$$-\int_0^T\e^{-\frac{t}{\beta}}(\B w^0,e\varphi(t))\de t,$$
we get \eqref{wweak}. This, together with \eqref{ini*}, shows that $(u,w)\in \V\times  H^1(0,T;H^d_s)$ is a weak solution to the coupled system \eqref{cs}--\eqref{csi} in the sense of Definition \ref{defsym}. The proof is then complete.
\end{proof}

Now we are in position to state the main result of this section.
\begin{theorem}\label{main2}
	There exists a weak solution $(u,w)\in\V\times  H^1(0,T;H^d_s)$ to the coupled system \eqref{cs}--\eqref{csi} according to Definition \ref{defsym}. Moreover, $u\in C^0_w([0,T];V)$, $\dot u\in C^0_w([0,T];H)\cap H^1(0,T;(V_0^D)')$, and
	\begin{equation*}
  \lim_{t\to 0^+}u(t)= u^0\text{ in $V$}\quad\text{and}\quad\lim_{t\to 0^+}\dot u(t)=u^1\text{ in $H$}.
	\end{equation*}
\end{theorem}

The proof of this result will be given at the end of this section.

\subsection{Discretization in time}\label{sub3}
In this subsection we prove Theorem \ref{main2} by means of a time discretization scheme in the same spirit of~\cite{DM-Lar}. 

Let us fix $n\in\N$ and set
\begin{equation}\label{prime-def}
    \tau_n:=\frac{ T}{n},\quad u_n^0:=u^0,\quad u_n^{-1}:=u^0-\tau_nu^1,\quad w_n^0:=w^0,\quad F_n^0:=F(0),\quad h_n^0:=\B w^0.
\end{equation}
We define 
\begin{alignat}{4}
&V_n^k:=V^D_{k\tau_n},&&\quad z_n^k:=z(k\tau_n)  &&\quad &&\quad\text{for }k=0,\dots,n,\nonumber\\
&F^k_n:=F(k\tau_n),&&\quad h_n^k:=\e^{-\frac{k\tau_n}{\beta}}\B w^0,&&\quad f_n^k:=\dashint_{(k-1)\tau_n}^{k\tau_n} f(\tau)\de \tau
&&\quad\text{for $k=1,\dots,n$}.\nonumber
\end{alignat}

For $k=1,...,n$ let $(u_n^k,w_n^k)$ be the minimizer in $V_n^k\times H^d_s$ of the functional
\begin{align}\label{func}
    (u,w)\mapsto &\frac{1}{2\tau_n^2}\|u-2u_n^{k-1}+u_n^{k-2}\|^2+\frac{1}{2}(\A eu,eu)+\frac{1}{2}(\B (eu-w),eu-w)\nonumber\\
    &+\frac{\beta}{2\tau_n^2}(\B (w-w_n^{k-1}),w-w_n^{k-1})-(f_n^k,u)-(F_n^k-h_n^k,eu).
\end{align}
Using the coerciveness \eqref{CB2M}, it is easy to see that the functional in \eqref{func} is convex and bounded from below by 
$$\frac{1}{4}\min\Big\{\frac{1}{2\tau_n^2},C_{\A},\frac{1}{\tau_n^2}C_{\B}\beta\Big\}(\|u\|^2_V+\|w\|^2)-C_n^k,$$ 
for a suitable positive constant $C_n^k$. The existence of a minimizer then follows from the lower semicontinuity of the functional with respect to the strong (and hence to the weak) convergence in $V_n^k\times H^d_s$.

To simplify the exposition, for $k= 0,...,n$ we define 
\begin{equation}\label{derivate}
    \delta u_n^k:=\frac{u_n^k-u_n^{k-1}}{\tau_n}\quad \text{and}\quad  \delta^2 u_n^k:=\frac{\delta u_n^k-\delta u_n^{k-1}}{\tau_n}.
\end{equation}
The Euler equation for \eqref{func} gives
\begin{align}\label{unkM}
(\delta^2u_n^k,\varphi)+(\A  eu_n^k,e\varphi)&+(\B(eu_n^k-w_n^k),e\varphi-\psi)+\beta(\B\delta w_n^k,\psi)=(f_n^k,\varphi)+(F_n^k-h_n^k,e\varphi),
\end{align}
for every $(\varphi,\psi)\in V_n^k\times H^d_s$, where $\delta w_n^k$ is defined for every $k=1,\dots,n$ as in \eqref{derivate}, and $\delta u_n^0=u^1$ by \eqref{prime-def}.
Notice that by choosing as a test function the pair $(\varphi,0)$ with $\varphi\in V_n^k$, we get
$$(\delta^2u_n^k,\varphi)+((\A +\B)  eu_n^k-\B w_n^k,e\varphi)=(f_n^k,\varphi)+(F_n^k-h_n^k,e\varphi),$$
which is a discrete-in-time approximation of \eqref{wweak}. On the other hand, if we use as a test function in \eqref{unkM} the pair $(0,\psi)$ with $\psi\in H^d_s$, we have
$$(\beta\delta w_n^k+w_n^k-eu_n^k,\psi)=0,$$
whence $\beta\delta w_n^k+w_n^k-eu_n^k=0$ (as element of $H^d_s$), which is an approximation in time of \eqref{odes}. 

In the next lemma we show an estimate for the family $\{(u_n^k,w_n^k)\}_{k=1}^n$, which is uniform with respect to $n$, and it will be used later to pass to the limit in the discrete equation \eqref{unkM}.

\begin{lemma}\label{lem:estM}
There exists a positive constant $C$, independent of $n$, such that 
\begin{equation}\label{eq:estM}
\max_{i=1,..,n}\|\delta u_n^i\|+\max_{i=1,..,n}\norm{eu_n^i}+\max_{i=1,..,n}\norm{w_n^i}+\sum_{i=1}^n \tau_n \norm{\delta w_n^i}^2\le C.
\end{equation}
\end{lemma}

\begin{proof}
To simplify our computations, we define the following two bilinear symmetric forms
\begin{alignat*}{2}
&a:  (V\times H^d_s)\times (V\times H^d_s)\rightarrow\R  \qquad &&\qquad b: H^d_s\times H^d_s\rightarrow \R\\  
&a((u,w),(\varphi,\psi)):=(\A  eu,e\varphi)+( \B(eu-w),e\varphi-\psi), &&\qquad b(w,\psi):=\beta(\B w,\psi).
\end{alignat*}
Thanks to \eqref{CB2M} we have that $a((\varphi,\psi),(\varphi,\psi))\geq 0$ and $b(\psi,\psi)\geq 0$ for every $\varphi\in V$ and $\psi\in H^d_s$. Now we set $\omega_n^k:=(u_n^k,w_n^k)$ for $k=0,\dots,n$, and we take $(\varphi,\psi)=\tau_n(\delta u_n^k-\delta z_n^k,\delta w_n^k)\in V_n^k\times H^d_s$ as a test function in \eqref{unkM}, where $\delta z_n^0:=\dot z(0)$ and $\delta z_n^k$ is defined as in \eqref{derivate}. Therefore, we obtain
\begin{align}
\|\delta u_n^k\|^2-(\delta u_n^{k-1},\delta u_n^{k})&-\tau_n(\delta^2 u_n^k,\delta z^k_n)+a(\omega_n^k,\omega_n^k)-a(\omega_n^{k-1},\omega_n^k)-\tau_n a(\omega_n^k,(\delta z_n^k,0))+\tau_n b(\delta w_n^k,\delta w_n^k)\nonumber\\
&=\tau_n(f_n^k,\delta u_n^k-\delta z_n^k)+\tau_n(F_n^k,e\delta u_n^k-e\delta z_n^k)- \tau_n(h_n^k,e\delta u_n^k-e\delta z_n^k).\label{us}
\end{align}
By means of the following identities
\begin{align*}
&\norm{\delta u_n^k}^2-(\delta u_n^{k-1},\delta u_n^k)=\frac{1}{2}\norm{\delta u_n^k}^2-\frac{1}{2}\norm{\delta u_n^{k-1}}^2+\frac{\tau_n^2}{2}\norm{\delta^2 u_n^k}^2,\\
&a(\omega_n^k,\omega_n^k)-a(\omega_n^{k-1},\omega_n^k)=\frac{1}{2}a(\omega_n^k,\omega_n^k)-\frac{1}{2}a(\omega_n^{k-1},\omega_n^{k-1})+\frac{\tau_n^2}{2}a(\delta \omega_n^k,\delta \omega_n^k),
\end{align*}
from \eqref{us} we infer
\begin{equation}\label{nmr}
\frac{1}{2}\|\delta u_n^k\|^2-\frac{1}{2}\|\delta u_n^{k-1}\|^2+\frac{1}{2}a(\omega_n^k,\omega_n^k)-\frac{1}{2}a(\omega_n^{k-1},\omega_n^{k-1})+\tau_n b(\delta w_n^{k},\delta w_n^{k})\leq\tau_n W_n^k,
\end{equation}
where
\begin{align*}
W_n^k:=(f_n^k,\delta u_n^k-\delta z_n^k)+(F_n^k,e\delta u_n^k-e\delta z_n^k)- (h_n^k,e\delta u_n^k-e\delta z_n^k)+(\delta^2 u_n^k,\delta z^k_n)+ a(\omega_n^k,(\delta z_n^k,0)).   
\end{align*}
We fix $i\in\{1,\dots,n\}$ and we sum in \eqref{nmr} over $k= 1,\dots,i$ to obtain the following discrete energy inequality
\begin{equation}\label{grwM}
\frac{1}{2}\|\delta u_n^i\|^2+\frac{1}{2}a(\omega_n^i,\omega_n^i)+\sum_{k=1}^i\tau_n b(\delta w_n^{k},\delta w_n^{k})\leq \En_0+\sum_{k=1}^i \tau_n W_n^k,
\end{equation}
where 
$$\En_0:=\frac{1}{2}\norm{u^1}^2+\frac{1}{2}(\A  eu^0,eu^0)+\frac{1}{2}(\B(eu^0-w^0),eu^0-w^0).$$ 
Let us now estimate the right-hand side of \eqref{grwM} from above. By means of Cauchy-Schwarz and Young's inequalities we can write
\begin{align}
\Big|\sum_{k=1}^i \tau_n(f_n^k,\delta u_n^k-\delta z_n^k)\Big| 
&\leq \|f\|^2_{L^2(0,T;H)}+\frac{1}{2}\|\dot z\|^2_{L^2(0,T;H)}+\frac{1}{2}\sum_{k=1}^i \tau_n\|\delta u_n^k\|^2,\\
\Big|\sum_{k=1}^i \tau_n(h_n^k,\delta z_n^k)\Big| 
&\leq \frac{1}{2}\sum_{k=1}^i\tau_n\e^{-\frac{2k\tau_n}{\beta}}\|\B w^0\|^2+\frac{1}{2}\sum_{k=1}^i\tau_n\|\delta z_n^k\|^2\leq \frac{ T}{2}\|\B w^0\|^2+\frac{1}{2}\|\dot z\|^2_{L^2(0,T;H)},\\
\Big|\sum_{k=1}^i \tau_n(F_n^k,e\delta z_n^k)\Big| &\leq \frac{1}{2}\sum_{k=1}^i\tau_n\|F_n^k\|^2+\frac{1}{2}\sum_{k=1}^i\tau_n\|e\delta z_n^k\|^2\nonumber\\
&\leq T\|F(0)\|^2+ T^2\|\dot F\|^2_{L^2(0,T;H^d_s)}+\frac{1}{2}\|\dot z\|^2_{L^2(0,T;V)},\\
\Big|\sum_{k=1}^i\tau_n a(\omega_n^k,(\delta z_n^k,0))\Big|&\leq \frac{1}{2}\|\A  \|^2_{{\infty}}\sum_{k=1}^i \tau_n\|eu_n^k\|^2+\frac{1}{2}\|\B\|^2_{{\infty}}\sum_{k=1}^i \tau_n\|eu^k_n-w_n^k\|^2+\sum_{k=1}^i \tau_n\|e\delta z_n^k\|^2\nonumber\\
&\leq \frac{1}{2}(\|\A\|^2_{{\infty}}+\|\B\|^2_{{\infty}})\Big[\sum_{k=1}^i \tau_n\|eu_n^k\|^2+\sum_{k=1}^i \tau_n\|eu_n^k-w_n^k\|^2\Big]+\|\dot z\|^2_{L^2(0,T;V)}.
\end{align}
Notice that the following discrete integrations by parts hold
\begin{align}
    &\sum_{k=1}^i \tau_n(\delta^2 u_n^k,\delta z^k_n)=(\delta u_n^i,\delta z_n^i)-(\delta u_n^0,\delta z_n^0)-\sum_{k=1}^i\tau_n (\delta u_n^{k-1},\delta^2 z_n^k),\label{dis-part}\\
    &\sum_{k=1}^i \tau_n(h_n^k,e\delta u^k_n)=(h_n^i,eu_n^i)-(h_n^0,eu_n^0)-\sum_{k=1}^i\tau_n (\delta h_n^{k},eu_n^{k-1}),\label{dis-part2}\\
    &\sum_{k=1}^i \tau_n(F_n^k,e\delta u^k_n)=(F_n^i,eu_n^i)-(F_n^0,eu^0)-\sum_{k=1}^i\tau_n (\delta F_n^{k},eu_n^{k-1}),\label{dis-part2.1}
\end{align}
where $\delta h_n^k$, $\delta F_n^k$, and $\delta^2 z_n^k$ are defined as in \eqref{derivate}. By \eqref{dis-part} and
\begin{equation}\label{ti2}
    \sum_{k=1}^i \tau_n \|\delta u_n^{k-1}\|^2=\sum_{k=0}^{i-1} \tau_n \|\delta u_n^k\|^2\leq T\|u^1\|^2+\sum_{k=1}^i \tau_n \|\delta u_n^k\|^2,
\end{equation}
we can write for every $\varepsilon_1>0$
\begin{align}
        \Big| \sum_{k=1}^i \tau_n(\delta^2 u_n^k,\delta z^k_n)\Big|&\leq \frac{1}{2\varepsilon_1}\|\delta z_n^i\|^2+\frac{\varepsilon_1}{2}\|\delta u_n^i\|^2+\|u^1\|\|\dot z(0)\|+ \sum_{k=1}^i \tau_n \|\delta u_n^{k-1}\|\|\delta^2 z_n^k\|\nonumber\\
    &\leq  C_{\varepsilon_1}+\|\ddot z\|^2_{L^2(0,T;H)}+\frac{\varepsilon_1}{2}\|\delta u_n^i\|^2+\frac{1}{2}\sum_{k=1}^i \tau_n \|\delta u_n^{k}\|^2,
\end{align}
where $C_{\varepsilon_1}$ is a positive constant depending on $\varepsilon_1$. Thanks to \eqref{dis-part2} and to \eqref{ti2} (applied to $eu_n^{k-1}$ in place of $\delta u_n^{k-1}$) we have for every $\varepsilon_2>0$
\begin{align}
    \Big|\sum_{k=1}^i \tau_n(h_n^k,e\delta u^k_n)\Big|&\leq \frac{1}{2\varepsilon_2}\|h_n^i\|^2+\frac{\varepsilon_2}{2}\|e u_n^i\|^2+\|eu^0\|\|\B w^0\|+ \sum_{k=1}^i \tau_n \|\delta h_n^{k}\|\|eu_n^{k-1}\|\nonumber\\
    &\leq C_{\varepsilon_2}+\frac{1}{2\beta}\|\B w^0\|^2+\frac{\varepsilon_2}{2}\|e u_n^i\|^2+\frac{1}{2}\sum_{k=1}^i \tau_n \|e u_n^{k}\|^2,
\end{align}
where $C_{\varepsilon_2}$ is a positive constant depending on $\varepsilon_2$. Moreover, notice that
$$u_n^i=\sum_{k=1}^i\tau_n\delta u_n^k+u^0,$$
hence by means of the discrete Holder's inequality
\begin{equation}\label{fudev}
    \norm{u_n^i}\leq \sum_{k=1}^i\tau_n \norm{\delta u_n^k}+\norm{u^0}\leq  T^{\frac{1}{2}}\Big(\sum_{k=1}^i\tau_n\norm{\delta u_n^k}^2\Big)^{\frac{1}{2}}+\norm{u^0}.
\end{equation}
By \eqref{dis-part2.1}, \eqref{ti2} (applied again to $eu_n^{k-1}$ in place of $\delta u_n^{k-1}$), and \eqref{fudev} we get for every $\varepsilon_3>0$
\begin{align}
   \Big| \sum_{k=1}^i \tau_n(F_n^k,e\delta u^k_n)\Big|&\leq \frac{1}{2\varepsilon_3 }\|F_n^i\|^2+\frac{\varepsilon_3 }{2}\|eu_n^i\|^2+\|F(0)\|\|eu^0\|+\sum_{k=1}^i\tau_n \|\delta F_n^{k}\|\|eu_n^{k-1}\|\nonumber\\
   &\leq C_{\varepsilon_3 }+\frac{\varepsilon_3 }{2}\|eu_n^i\|^2+ \frac{1}{2}\|\dot F\|^2_{L^2(0,T;H^d_s)}+\frac{1}{2}\sum_{k=1}^i\tau_n \|eu_n^{k}\|^2,\label{stime*}
\end{align}
where $C_{\varepsilon_3}$ is a positive constant depending on $\varepsilon_3$.

Now we consider \eqref{grwM}--\eqref{stime*}. By choosing  $\varepsilon_1=\frac{1}{2}$, $\varepsilon_2=\varepsilon_3=\frac{C_{\A  }}{4}$ and using \eqref{CB1.5} and \eqref{CB2M} we obtain the existence of two positive constants $C_1$ and $C_2$ such that
\begin{align}\label{gro*}
\frac{1}{4}\|\delta u_n^i\|^2+\frac{C_{\A  }}{4}\|eu_n^i\|^2&+\frac{C_{\B}}{2}\|eu_n^i-w_n^i\|^2+\beta C_{\B}\sum_{k=1}^i\tau_n \|\delta w_n^k\|^2\nonumber \\
&\leq C_1+C_2\sum_{k=1}^i\tau_n\Big[\|\delta u_n^k\|^2+\|eu_n^k\|^2+\|eu_n^k-w_n^k\|^2+\sum_{l=1}^k\tau_n\|\delta w_n^l\|^2\Big].
\end{align}
By defining
$$a_n^i:=\|\delta u_n^i\|^2+\norm{eu_n^i}^2+\norm{eu_n^i-w_n^i}^2+\sum_{k=1}^i\tau_n \norm{\delta w_n^k}^2,$$
from \eqref{gro*} we can derive
$$a_n^i\leq \tilde{C}_1+\tilde{C}_2\sum_{k=1}^i \tau_n a_n^k,$$
for two positive constants $\tilde{C}_1$ and $\tilde{C}_2$. Taking into account a discrete version of Gronwall's lemma (see, e.g., \cite[Lemma 3.2.4]{AGS}) we deduce that $a_n^i$ is bounded by a positive constant $C^*$ independent of $i$ and $n$; i.e. 
\begin{equation*}
\|\delta u_n^i\|^2+\norm{eu_n^i}^2+\norm{eu_n^i-w_n^i}^2+\sum_{k=1}^i\tau_n\norm{\delta w_n^k}^2\leq C^*\qquad \text{for every $i=1,\dots,n$ and for every $n\in\N$}.
\end{equation*}
Therefore 
\begin{equation*}
\|\delta u_n^i\|^2+\norm{eu_n^i}^2+\norm{w_n^i}^2+\sum_{k=1}^i\tau_n\norm{\delta w_n^k}^2\leq 3C^*\qquad \text{for every $i=1,\dots,n$ and for every $n\in\N$},
\end{equation*}
and this concludes the proof.
\end{proof}

We now want to pass to the limit into the discrete equation \eqref{unkM} to obtain a solution to the coupled system \eqref{cs}--\eqref{csi} according to Definition \ref{defsym}. We start by defining the following interpolation sequences of our limit solution
\begin{align*}
&u_n(t):=u_n^k+(t-k\tau_n)\delta u_n^k, &&\tilde{u}_n(t):=\delta u_n^k+(t-k\tau_n)\delta^2 u_n^k && t\in [(k-1)\tau_n,k\tau_n],\quad k=1,\dots,n,\\
&u^+_n(t):=u_n^k, &&\tilde{u}^+_n(t):=\delta u_n^k && t\in ((k-1)\tau_n,k\tau_n], \quad k=1,\dots,n,\\
&u^-_n(t):=u_n^{k-1}, &&\tilde{u}^-_n(t):=\delta u_n^{k-1}&& t\in [(k-1)\tau_n,k\tau_n), \quad k=1,\dots,n,
\end{align*}
and the same approximations $w_n,w_n^+,w_n^-$ for the function $w$. By using this notation, we can state the following convergence lemma.
\begin{lemma}\label{lem:convM}
There exists $(u,w)\in\V\times H^1(0,T;H^d_s)$, with $u-z\in \V^D$, such that, up to a not relabeled subsequence, we have
\begin{alignat}{4}
&u_n \xrightharpoonup[n\to \infty]{H^1(0, T;H)}u,&&\quad u^\pm_n \xrightharpoonup[n\to \infty]{L^2(0, T;V)}u,&&\quad \tilde{u}^\pm_n \xrightharpoonup[n\to \infty]{L^2(0, T;H)}\dot{u},\label{pf1M}\\
&w_n \xrightharpoonup[n\to \infty]{H^1(0, T;H^d_s)}w,&&\quad w^\pm_n \xrightharpoonup[n\to \infty]{L^2(0, T;H^d_s)}w.\label{pf2M}
\end{alignat}
\end{lemma}
\begin{proof}
Thanks to Lemma \ref{lem:estM} the sequences 
\begin{align*}
    &\{u_n\}_n\subset H^1(0, T;H)\cap L^\infty(0, T;V),&& \{w_n\}_n\subset H^1(0, T;H^d_s)\cap L^\infty(0, T;H^d_s),\\
    &\{u_n^\pm\}_n\subset L^\infty(0, T;V),&&\{w_n^\pm\}_n\subset L^\infty(0, T;H^d_s),\\
    &\{\tilde u_n^\pm\}_n\subset L^\infty (0, T;H),
\end{align*}
are uniformly bounded.
Indeed, by means of \eqref{eq:estM} and \eqref{fudev} there exists a positive constant $\bar{C}$ such that $\|u_n^i\|_{V}\leq \bar C$ for every $n\in\N$ and $i=1,..,n$, and therefore
\begin{equation*}
\|u_n\|_{L^{\infty}(0, T;V)}\leq \max_{k=1,..,n}\ \sup_{t\in [(k-1)\tau_n,k\tau_n]}\|\left(1-k+t\tau^{-1}_n\right)u_n^k+\left(k-t\tau^{-1}_n\right)u_n^{k-1}\|_{V}\leq 2\bar C.
\end{equation*} 
By Banach-Alaoglu's Theorem there exist some functions 
$$u\in H^1(0, T;H),\quad w\in H^1(0, T;H^d_s),\quad v_1\in L^2(0, T;V),\quad v_2\in L^2(0, T;H^d_s)$$ 
such that, up to a not relabeled sequence, we have
\begin{alignat}{4}
&u_n \xrightharpoonup[n\to \infty]{L^2(0, T;V)}u, \qquad &&\dot{u}_n \xrightharpoonup[n\to \infty]{L^2(0, T;H)}\dot{u}, \qquad &&u^+_n \xrightharpoonup[n\to \infty]{L^2(0, T;V)}v_1,\label{lim1}\\
 &w_n\xrightharpoonup[n\to \infty]{L^2(0, T;H^d_s)}w,\qquad  &&\dot{w}_n \xrightharpoonup[n\to \infty]{L^2(0, T;H^d_s)}\dot{w},\qquad  &&w^+_n \xrightharpoonup[n\to \infty]{L^2(0, T;H^d_s)}v_2.\label{lim2}
\end{alignat}
Since there exists a positive constant $C$ such that
\begin{align}\label{nuova}
\|u_n-u^+_n\|_{L^\infty(0, T;H)}\leq C \tau_n\xrightarrow[n\to\infty]{}0, \quad \|w_n-w^+_n\|_{L^\infty(0, T;H^d_s)}\leq C \tau_n\xrightarrow[n\to\infty]{}0,
\end{align}
by using \eqref{lim1}, \eqref{lim2} and triangle inequality, we can conclude that $u=v_1$ and $w=v_2$.\\ 
Moreover, given that
\begin{alignat*}{4}
&u^-_n(t)=u^+_n(t-\tau_n), && \qquad w^-_n(t)=w^+_n(t-\tau_n) &&\quad \text{for } t\in(\tau_n, T),\\
&\tilde u^-_n(t)=\tilde u^+_n(t-\tau_n), &&\qquad
&&\quad \text{for } t\in(\tau_n, T),\\
&\tilde{u}^+_n(t)=\dot{u}_n(t), &&
&&\quad\text{for a.e. } t\in(0, T),
\end{alignat*}
with \eqref{nuova} and the continuity of the translations in $L^2$ we deduce that 
\begin{equation*}
u^-_n \xrightharpoonup[n\to \infty]{L^2(0, T;V)} u,\quad \tilde{u}^\pm_n \xrightharpoonup[n\to \infty]{L^2(0, T;H)}\dot{u}, \quad w^-_n \xrightharpoonup[n\to \infty]{L^2(0, T;H^d_s)} w.
\end{equation*}

Now let us check that $u\in\V$. To this aim, we define the following sets
\begin{align*}
\tilde {\V}&:=\{u\in L^2(0, T;V): u(t)\in V_t\hspace{2pt}\text{ for a.e. $t\in (0, T)$}\}\subset L^2(0, T;V),\\
\tilde {\V}^D&:=\{u\in\tilde {\V}: u(t)\in V^D_t\text{ for a.e. $t\in (0, T)$}\}\subset L^2(0, T;V).
\end{align*}
Notice that $\tilde {\V}$ is a (strong) closed convex subset of $L^2(0, T;V)$, and so by Hahn-Banach Theorem the set $\tilde {\V}$ is weakly closed. In the same way we can prove that $\tilde {\V}^D$ is also a weakly closed set. Notice that $\{u^-_n\}_n\subset \tilde {\V}$, indeed
\begin{align*}
u^-_n(t)=u^{k-1}_n\in V_{(k-1)\tau_n}\subset V_t\quad\text{for $t\in [(k-1)\tau_n,k\tau_n)$, $k=1,\dots,n$}.
\end{align*}
Since $u^-_n\xrightharpoonup[n\to\infty]{L^2(0, T;V)} u$, we conclude that $u\in \tilde {\V}$. Moreover $\tilde u^+_n\xrightharpoonup[n\to\infty]{L^2(0, T;H)} \dot u$ and so $\dot u\in L^2(0, T;H)$, from which we have $u\in \V$. Finally, to show that $u-z\in \V^D$ we observe that
\begin{align*}
u_n^-(t)-z_n^-(t)=u^{k-1}_n-z^{k-1}_n\in V_n^{k-1}\subset V_t^D\quad \text{ for $t\in[(k-1)\tau_n,k\tau_n)$, $k=1,\dots,n$},
\end{align*}
therefore $\{u_n^--z_n^-\}_n\subset \tilde {\V}^D$. Since 
$$u_n^-\xrightharpoonup[n\to\infty]{L^2(0, T;V)}u,\quad z_n^-\xrightarrow[n\to\infty]{L^2(0, T;V)}z,$$ 
we get $u-z\in\V^D$. This concludes the proof.
\end{proof}

With the next lemma we show that the limit identified by Lemma \ref{lem:convM} is actually a weak solution to the coupled system \eqref{cs}--\eqref{csi} according to Definition \ref{defsym}.
\begin{lemma}\label{limitfun}
The limit pair $(u,w)\in\V\times H^1(0,T;H^d_s)$ of Lemma \ref{lem:convM} satisfies \eqref{wweak} and \eqref{odes}.
\end{lemma}
\begin{proof}
We fix $n\in\N$ and the functions $\varphi\in \D$ and $\psi\in C_c^{\infty}(0, T;H^d_s)$. We consider the following piecewise-constant approximating sequences 
\begin{alignat*}{3}
\varphi_n^k&:=\varphi(k\tau_n)\qquad\qquad\psi_n^k&&:=\psi(k\tau_n)&&\qquad\text{for $k=0,\dots,n$,}\\    
\delta \varphi_n^k&:=\frac{\varphi_n^k-\varphi_n^{k-1}}{\tau_n}\qquad\delta \psi_n^k&&:=\frac{\psi_n^k-\psi_n^{k-1}}{\tau_n}&&\qquad \text{for }k=1,\dots,n,
\end{alignat*}
and the approximating sequences
\begin{align*}
&\varphi^+_n(t):=\varphi_n^k, & & \tilde\varphi^+_n(t):=\delta\varphi_n^k & & t\in ((k-1)\tau_n,k\tau_n], \quad k=1,\dots,n,\\
&\psi^+_n(t):=\psi_n^k, & & \tilde\psi^+_n(t):=\delta\psi_n^k & & t\in ((k-1)\tau_n,k\tau_n], \quad k=1,\dots,n.
\end{align*}
If we use $\tau_n(\varphi_n^k,0)\in V_n^k\times H^d_s$ as a test function in \eqref{unkM}, after summing over $k=1,...,n$, we get
\begin{align}\label{limitM}
\sum_{k=1}^n\tau_n(\delta^2u_n^k,\varphi^k_n)&+\sum_{k=1}^n\tau_n((\A +\B)  eu_n^k-\B w_n^k,e\varphi^k_n)\nonumber\\
&=\sum_{k=1}^n\tau_n(f_n^k,\varphi^k_n)+\sum_{k=1}^n\tau_n(F_n^k,e\varphi^k_n)-\sum_{k=1}^n\tau_n(h_n^k,e\varphi^k_n).
\end{align}
Since $\varphi_n^0=\varphi_n^n=0$ we obtain
\begin{align*}
\sum_{k=1}^n \tau_n(\delta^2 u^k_n,\varphi^k_n)
&=\sum_{k=1}^{n} (\delta u^k_n,\varphi^k_n)-\sum_{k=1}^n (\delta u^{k-1}_n,\varphi^k_n)=\sum_{k=0}^{n-1} (\delta u^k_n,\varphi^k_n)-\sum_{k=0}^{n-1} (\delta u^{k}_n,\varphi^{k+1}_n)\\
&=-\sum_{k=0}^{n-1} \tau_n (\delta u^k_n,\delta \varphi^{k+1}_n)=-\sum_{k=1}^n\tau_n(\delta u^{k-1}_n,\delta \varphi^k_n)=-\int_0^{ T}(\tilde{u}^-_n(t),\tilde{\varphi}^+_n(t)) \de t,
\end{align*}
and from \eqref{limitM} we deduce
\begin{align}\label{eqappM}
-\int_0^{ T}(\tilde{u}^-_n(t),\tilde{\varphi}^+_n(t)) \de t&+\int_0^{ T}((\A +\B) e u^+_n(t)-\B w_n^+(t),e \varphi^+_n(t)) \de t\nonumber\\
&=\int_0^{ T}(f^+_n(t),\varphi^+_n(t)) \de t+\int_0^{ T}(F^+_n(t),e\varphi^+_n(t)) \de t-\int_0^{ T}(h^+_n(t),e\varphi^+_n(t)) \de t.
\end{align}
Thanks to \eqref{pf1M}, \eqref{pf2M}, and the convergences 
\begin{equation*}
\varphi^+_n\xrightarrow[n\to\infty]{L^2(0,{ T};V)}\varphi, \quad \tilde{\varphi}^+_n\xrightarrow[n\to\infty]{L^2(0,{ T};H)}\dot{\varphi}
\end{equation*}
we can pass to the limit in \eqref{eqappM}, and we get that $u\in\V$ satisfies \eqref{wweak} for every function $\varphi\in \D$. 

If we use $\tau_n(0,\psi_n^k)\in V_n^k\times H^d_s$ as a test function in \eqref{unkM}, we have
\begin{equation*}
(\beta\delta w_n^k+w_n^k-eu_n^k,\psi^k_n)=0,
\end{equation*}
which corresponds to
$$(\beta\dot w_n(t)+w_n^{+}(t)-eu_n^{+}(t),\psi^{+}_n(t))=0\qquad t\in ((k-1)\tau_n,k\tau_n], \quad k=1,\dots,n.$$
Therefore, for every $(a,b)\subset(0, T)$, from \eqref{pf1M} and \eqref{pf2M}, we can write
\begin{equation}\label{lim2eq}
0=\lim_{n\to\infty}\dashint_a^b  (\beta\dot w_n(t)+w_n^{+}(t)-eu_n^{+}(t),\psi^{+}_n(t))\de t  =\dashint_a^b  (\beta\dot w(t)+w(t)-eu(t),\psi(t))\de t.
\end{equation}
Now we pass to the limit in \eqref{lim2eq} as $a\to b$ and  we obtain
$$(\beta\dot w(b)+w(b)-eu(b),\psi(b))=0\qquad \text{for every $b\in [0, T]$.}$$
Given that, fixed $b\in(0, T)$ for every $\textbf{p}\in H^d_s$ there exists $\psi_{\textbf{p}}(t):=(t+1-b)\textbf{p}\in H^1(0,T;H^d_s)$ such that $\psi_{\textbf{p}}(b)=\textbf{p}$, we can say that for a.e. $t\in(0, T)$ we have $\beta\dot w(t)+w(t)-eu(t)=0$ in $H^d_s$. Finally, since $w_n(0)=w^0$, taking into account \eqref{pf2M} we can conclude that $w(0)=w^0$.
\end{proof}
 
It remains to show that the limit previously found assumes the initial data in the sense of \eqref{ini*}. Before doing this, let us recall the following result, whose proof can be found for example in~\cite{DL}.

\begin{lemma}\label{lem:wcM}
Let $X,Y$ be reflexive Banach spaces such that $X\hookrightarrow Y$ continuously. Then 
$$L^{\infty}(0, T;X)\cap C^0_w([0, T];Y)=  C^0_w([0, T];X).$$
\end{lemma}

\begin{proposition}\label{prop-main}
The limit pair $(u,w)\in \mathcal{V}\times  H^1(0,T;H^d_s)$ of Lemma \ref{lem:convM} is a weak solution to the coupled system \eqref{cs}--\eqref{csi}. Moreover, $u\in C^0_w([0,T];V)$, $\dot u\in C^0_w([0,T];H)$ and it admits a distributional derivative in the space $L^2(0,T;(V_0^D)')$.
\end{proposition}

\begin{proof}
From the discrete equation \eqref{unkM} we deduce
\begin{equation*}
|(\delta^2 u^k_n,\varphi)|\leq  \|\A  \|_{{\infty}}\|e u^k_n\|+\|\B\|_{{\infty}}\|eu_n^k-w_n^k\|+\beta\|\B\|_{\infty}\|\delta w^k_n\|+\| f^k_n\|+\|F_n^k\|+\| h^k_n\|,
\end{equation*}
 for every $(\varphi,\psi)\in V^D_0\times H^d_s \subset V_n^k\times H^d_s$ such that $\|(\varphi,\psi)\|_{V\times H^d_s}\leq 1$. Therefore, taking the supremum over $(\varphi,\psi)\in V^D_0\times H^d_s$ with $\norm{(\varphi,\psi)}_{V\times H^d_s}\leq1$, we obtain the existence of a positive constant $C'$ such that 
\begin{equation*}
\|\delta^2 u^k_n\|_{(V^D_0)'}^2\leq C'(\|e u^k_n\|^2+\|eu_n^k-w_n^k\|^2+\|\delta w^k_n\|^2+\|f^k_n\|^2+\|F_n^k\|^2+\|h^k_n\|^2).
\end{equation*}
By multiplying this inequality by $\tau_n$ and then by summing over $k=1,\dots,n$, we get
\begin{align}\label{d2M}
&\sum_{k=1}^n\tau_n\|\delta^2 u^k_n\|^2_{(V^D_0)'}\leq C'\Big(\sum_{k=1}^n\tau_n\|e u^k_n\|^2+\sum_{k=1}^n\tau_n\|eu_n^k-w_n^k\|^2+\sum_{k=1}^n\tau_n\|\delta w^k_n\|^2+C''\Big),
\end{align}
where
$$C''\hspace{-2pt}:=\|f\|_{L^2(0,T;H)}^2+\|F\|_{L^2(0,T;H^d_s)}^2+T\|\B w^0\|^2.$$
Thanks to \eqref{d2M} and Lemma \ref{lem:estM} we conclude that there exists a positive constant $\tilde C$, which does not depend on $n$, such that
\begin{equation}\label{dev3}
    \sum_{k=1}^n\tau_n\|\delta^2 u^k_n\|^2_{(V^D_0)'}\leq\tilde C.
\end{equation}
In particular $\{\tilde u_n\}_n\subset H^1(0, T;(V^D_0)')$ is uniformly bounded (notice that $\dot{\tilde u}_n(t)=\delta^2 u_n^k$ for $t\in((k-1)\tau_n,k\tau_n)$ and $k=1,\dots,n$). Hence, up to extracting a further (not relabeled) subsequence  from the one of Lemma \ref{lem:convM}, we have
\begin{equation}\label{newM}
\tilde{u}_n\xrightharpoonup[n\to\infty]{H^1(0, T;(V^D_0)')}v,
\end{equation}
and by using the following estimate
\begin{equation*}
\|\tilde{u}_n-\tilde{u}^+_n\|^2_{L^2(0, T;(V^D_0)')}\leq\tilde C \tau^2_n\xrightarrow[n\to\infty]{}0,
\end{equation*}
we conclude that $v=\dot{u}$.

Since $H^1(0, T;(V^D_0)')\hookrightarrow C^0([0, T],(V^D_0)')$, by using Lemma \ref{lem:convM} and Lemma \ref{lem:wcM} we deduce that the limit pair $(u,w)\in \V\times  H^1(0,T;H^d_s)$ satisfies
 \begin{equation*}
  u\in C^0_w([0, T];V)\quad \text{and}\quad \dot{u}\in C^0_w([0, T];H).
\end{equation*}
By \eqref{pf1M} and \eqref{newM} we then obtain
\begin{equation}\label{prog}
u_n(t)\xrightharpoonup[n\to\infty]{H}u(t)\quad\text{and}\quad  \tilde{u}_n(t)\xrightharpoonup[n\to\infty]{(V^D_0)'}\dot{u}(t) \qquad\text{for every $t\in [0, T]$,}
\end{equation}
so that $u(0)=u^0$ and $\dot u(0)=u^1$, since $u_n(0)=u^0$ and $ {\tilde{u}}_n(0)=u^1$. By Lemma \ref{limitfun} we get the thesis.
\end{proof}

\subsection{Energy Estimate}\label{sub4} In this subsection, we prove an energy-dissipation inequality which holds for the weak solution $(u,w)\in\V\times H^1(0,T;H^d_s)$ to the coupled system \eqref{cs}--\eqref{csi}, provided by Lemma \ref{lem:convM}. Thanks to this, we are able to show the validity of the initial conditions in a stronger sense. The energy-dissipation inequality give us a relation among the mechanical energy, defined by the sum of kinetic and elastic energy, the dissipation and the total work exerted by external forces and by the boundary conditions. Therefore, let us define the total energy as
\begin{equation}\label{energy-mecc}
    \En_{u,w}(t):=\frac{1}{2}\norm{\dot u(t)}^2+\frac{1}{2}(\A  eu(t),eu(t))+\frac{1}{2}(\B(eu(t)-w(t)),eu(t)-w(t)).
\end{equation}
Notice that $\mathcal {E}_{u,w}(t)$ is well defined for every time $t\in[0, T]$ since $u\in C_w^0([0, T];V)$, $\dot u\in C_w^0([0, T];H)$ and $w\in C^0([0, T];H^d_s)$, and that 
\begin{equation*}
    \En_{u,w}(0)=\frac{1}{2} {\norm{ u^1}^2}+\frac{1}{2}(\A  eu^0,eu^0)+\frac{1}{2}(\B(eu^0-w^0),eu^0-w^0).
\end{equation*}
The dissipation, on the interval $[0,t]$, is defined by
\begin{equation}\label{dissi}
    \Dis_{u,w}(t):=\beta\int_0^t(\B \dot w(\tau),\dot w(\tau))\de \tau,
\end{equation}
and the total work is given by
\begin{align}\label{ttwork}
  \W_{tot}(t):&=\int_0^t [(f(\tau),\dot u(\tau)-\dot z(\tau))-(\dot F(\tau),eu(\tau)-ez(\tau))+((\A +\B)  eu(\tau)-\B w(\tau),e\dot z(\tau))]\de \tau\nonumber\\
  &\quad-\int_0^t(\dot u(\tau),\ddot z(\tau))\de \tau+(\dot{u}(t),\dot{z}(t)) \hspace{-1pt}-\hspace{-1pt}(u^1,\dot{z}(0))\hspace{-1pt}+\hspace{-1pt}(F(t),eu(t)-ez(t))\hspace{-1pt}-\hspace{-1pt}(F(0),eu^0-ez(0))\nonumber\\
  &\quad+\int_0^t\Big[\e^{-\frac{\tau}{\beta}}(\B w^0,e\dot z(\tau))-\frac{1}{\beta}\e^{-\frac{\tau}{\beta}}(\B w^0,eu(\tau))\Big]\de \tau-\e^{-\frac{t}{\beta}}(\B w^0,eu(t))+(\B w^0,eu^0).
\end{align}

\begin{remark}
From the classical point of view, the total work on the solution $(u,w)$ at time $t\in[0, T]$ is given by
\begin{equation}\label{totalwork2M}
    \W_C(t):=\W_{load}(t)+\W_{bdry}(t),
\end{equation}
where $\W_{load}(t)$ is the work on the solution at time $t\in[0, T]$ due to the loading term, which is defined as
\begin{equation}\label{cl1}
   \W_{load}(t):=\int_0^t(f(\tau), \dot u(\tau))\de \tau+\int_0^t(\div(\e^{-\frac{\tau}{\beta}}\B w^0-F(\tau)), \dot u(\tau))\de \tau,
\end{equation}
and $\W_{bdry}(t)$ is the work on the solution at time $t\in[0, T]$ due to the varying boundary conditions, which one expects to be equal to
\begin{align*}
   \W_{bdry}(t):&=\int_0^t((F_+(\tau)-\e^{-\frac{\tau}{\beta}}\B w^0_+)\nu,\dot u(\tau))_{L^2(\Gamma_{\tau})}\de \tau+\int_0^t((F_-(\tau)-\e^{-\frac{\tau}{\beta}}\B w^0_-)\nu,\dot u(\tau))_{L^2(\Gamma_{\tau})}\de \tau\\
   &+\int_0^t((F(\tau)-\e^{-\frac{\tau}{\beta}}\B w^0)\nu,\dot u(\tau))_{H^N}\de \tau+\int_0^t(((\A +\B) e u(\tau)-\B w(\tau))\nu,\dot z(\tau))_{H^D}\de \tau,
\end{align*}
where  $F_+(t)$, $w^0_+$ and $F_-(t)$, $w^0_-$ are the traces of $F(t)$ and $w^0$, respectively, from above and below on $\Gamma_t$.

Unfortunately, $\W_{load}(t)$ and $\W_{bdry}(t)$ are not well defined under our assumptions on $u$, $F$, and $w^0$. However, if we suppose more regularity, i.e., 
$$u\in H^1(0,T;H^2(\OG;\R^d))\cap H^2(0,T;H),\quad w\in H^1(0,T;H^1(\OG;\R^{d\times d}_{sym}))),\quad F\in H^1(0,T;H^1(\OG;\R_{sym}^{d\times d})),$$
$w^0\in V_0$, and that $\Gamma$ is a smooth manifold, then we can deduce from \eqref{wweak}, \eqref{odes}, and \eqref{ini*} that the pair $(u,w)$ satisfies 
\begin{equation}\label{newsys}
\left\{
\begin{alignedat}{4}
&\ddot u(t)-\div(\A   eu(t))-\div(\B( eu(t)-w(t)))=f(t)+\div(\e^{-\frac{t}{\beta}}\B w^0-F(t)) &&\quad &&\text{in $\OG_t$}, &&\quad\text{$t\in(0,T)$},\\
 &\beta \dot w(t)+w(t)-eu(t)=0 
\end{alignedat}
\right.
\end{equation}
with boundary and initial conditions
\begin{alignat*}{4}
&u(t)=z(t) && \qquad \text{on $\partial_D\Omega$}, && \quad t\in(0,T),\\
&[(\A +\B)  eu(t)-\B w(t)] \nu=[F(t)-\e^{-\frac{t}{\beta}}\B w^0]\nu && \qquad \text{on } \partial_N\Omega, &&\quad t\in(0,T),\\
&[(\A +\B)  eu_+(t)-\B w_+(t)] \nu=[F_+(t)-\e^{-\frac{t}{\beta}}\B w_+^0]\nu && \qquad \text{on } \Gamma_t, &&\quad t\in(0,T),\\
&[(\A +\B)  eu_-(t)-\B w_-(t)] \nu=[F_-(t)-\e^{-\frac{t}{\beta}}\B w_-^0]\nu && \qquad \text{on } \Gamma_t, &&\quad t\in(0,T),\\
& u(0)=u^0, \quad w(0)=w^0,\quad  \dot u(0)=u^1,
\end{alignat*}

In this case, $((\A +\B)  e u-w)\nu\in L^2(0,T;H^D)$ and by using~\eqref{newsys}, together with the divergence theorem and the integration by parts formula, we deduce
\begin{align}\label{cl}
&\int_0^t(((\A +\B)  e u(\tau)-\B w(\tau))\nu,\dot z(\tau))_{H^D}\de \tau=\int_0^t((\A +\B) e u(\tau)-\B w(\tau),e \dot z(\tau))\de \tau\nonumber\\
&\hspace{1cm}+\int_0^t\left[(\div((\A +\B)  e u(\tau)-\B w(\tau)),\dot z(\tau))+((\e^{-\frac{\tau}{\beta}}\B w^0-F(\tau))\nu,\dot z(\tau))_{H^N}\right]\de \tau\nonumber\\
&\hspace{1cm}+\int_0^t\left[((\e^{-\frac{\tau}{\beta}}\B w^0_+-F_+(\tau))\nu,\dot z(\tau))_{L^2(\Gamma_{\tau})}+((\e^{-\frac{\tau}{\beta}}\B w^0_--F_-(\tau))\nu,\dot z(\tau))_{L^2(\Gamma_{\tau})}\right]\de \tau\nonumber\\
&\hspace{0.1cm}=\int_0^t\left[((\A +\B) e u(\tau)-\B w(\tau),e \dot z(\tau))+((\e^{-\frac{\tau}{\beta}}\B w^0-F(\tau))\nu,\dot z(\tau))_{H^N}\right]\de \tau\nonumber\\
&\hspace{1cm}+\int_0^t\left[(\ddot u(\tau),\dot z(\tau))-(f(\tau),\dot z(\tau))+(\div F(\tau),\dot z(\tau))-\e^{-\frac{\tau}{\beta}}(\div(\B w^0),\dot{z}(\tau))\right]\de \tau\nonumber\\
&\hspace{1cm}+\int_0^t\left[((\e^{-\frac{\tau}{\beta}}\B w^0_+-F_+(\tau))\nu,\dot z(\tau))_{L^2(\Gamma_{\tau})}+((\e^{-\frac{\tau}{\beta}}\B w^0_--F_-(\tau))\nu,\dot z(\tau))_{L^2(\Gamma_{\tau})}\right]\de \tau\nonumber\\
&\hspace{0.1cm}=\int_0^t\left[((\A +\B) e u(\tau)-\B w(\tau),e \dot z(\tau))+((\e^{-\frac{\tau}{\beta}}\B w^0-F(\tau))\nu,\dot z(\tau))_{H^N}\right]\de \tau+(\dot{u}(t),\dot{z}(t)) -(u^1,\dot{z}(0))\nonumber\\
&\hspace{1cm}-\int_0^t\left[(\dot u(\tau),\ddot z(\tau))+(f(\tau),\dot z(\tau))-(\div F(\tau),\dot z(\tau))+\e^{-\frac{\tau}{\beta}}(\div(\B w^0),\dot{z}(\tau))\right]\de \tau\nonumber\\
&\hspace{1cm}+\int_0^t\left[((\e^{-\frac{\tau}{\beta}}\B w^0_+-F_+(\tau))\nu,\dot z(\tau))_{L^2(\Gamma_{\tau})}+((\e^{-\frac{\tau}{\beta}}\B w^0_--F_-(\tau))\nu,\dot z(\tau))_{L^2(\Gamma_{\tau})}\right]\de \tau.
\end{align}
From \eqref{cl} and the definition of $\W_{bdry}$, we have
\begin{align}\label{cl2}
\W_{bdry}(t)&=\int_0^t\left[((\A +\B) e u(\tau)-\B w(\tau),e \dot z(\tau))+((F(\tau)-\e^{-\frac{\tau}{\beta}}\B w^0)\nu,\dot{u}(\tau)-\dot z(\tau))_{H^N}\right]\de \tau \nonumber\\
&\hspace{0.5cm}-\int_0^t\left[(\dot u(\tau),\ddot z(\tau))+(f(\tau),\dot z(\tau))-(\div (F(\tau)-\e^{-\frac{\tau}{\beta}}\B w^0),\dot z(\tau))\right]\de \tau-(u^1,\dot{z}(0))+(\dot{u}(t),\dot{z}(t))\nonumber\\
&\hspace{0.5cm}+\int_0^t\left[((F_+(\tau)-\e^{-\frac{\tau}{\beta}}\B w^0_+)\nu,\dot{u}(\tau)-\dot z(\tau))_{L^2(\Gamma_{\tau})}+((F_-(\tau)-\e^{-\frac{\tau}{\beta}}\B w^0_-)\nu,\dot{u}(\tau)-\dot z(\tau))_{L^2(\Gamma_{\tau})}\right]\de \tau.
\end{align}
Taking into account \eqref{cl1} and \eqref{cl2}, the classical work \eqref{totalwork2M} can be written as 
\begin{align*}
    \W_C(t)&=\int_0^t\left[(f(\tau), \dot u(\tau)-\dot z(\tau))+((\A +\B) e u(\tau)-\B w(\tau),e \dot z(\tau))\right]\de \tau \\
    &\hspace{0.9cm}-\int_0^t(\dot{u}(\tau),\ddot{z}(\tau))\de \tau+(\dot{u}(t),\dot{z}(t))-(u^1,\dot{z}(0)) \\
    &\hspace{0.9cm}+\int_0^t\left[((F_+(\tau)-\e^{-\frac{\tau}{\beta}}\B w^0_+)\nu,\dot{u}(\tau)-\dot z(\tau))_{L^2(\Gamma_{\tau})}+((F_-(\tau)-\e^{-\frac{\tau}{\beta}}\B w^0_-)\nu,\dot{u}(\tau)-\dot z(\tau))_{L^2(\Gamma_{\tau})}\right]\de \tau\\
     &\hspace{0.9cm}-\int_0^t\left[(\div(F(\tau)-\e^{-\frac{\tau}{\beta}}\B w^0),\dot{u}(\tau)-\dot{z}(\tau))-((F(\tau)-\e^{-\frac{\tau}{\beta}}\B w^0)\nu,\dot{u}(\tau)-\dot z(\tau))_{H^N}\right]\de \tau\\
   &=\int_0^t\left[(f(\tau), \dot u(\tau)-\dot z(\tau))+((\A +\B) e u(\tau)-\B w(\tau),e \dot z(\tau))\right]\de \tau\\
     &\hspace{0.9cm}+\int_0^t\left[(F(\tau)-\e^{-\frac{\tau}{\beta}}\B w^0,e\dot{u}(\tau)-e\dot{z}(\tau))-(\dot{u}(\tau),\ddot{z}(\tau))\right]\de \tau+(\dot{u}(t),\dot{z}(t)) -(u^1,\dot{z}(0))\\
    &=\int_0^t\left[(f(\tau), \dot u(\tau)-\dot z(\tau))+((\A +\B) e u(\tau)-\B w(\tau),e \dot z(\tau))+\e^{-\frac{\tau}{\beta}}(\B w^0,e\dot{z}(\tau))\right]\de \tau\\
    &\hspace{0.9cm}-\int_0^t(\dot{u}(\tau),\ddot{z}(\tau))\de \tau+(\dot{u}(t),\dot{z}(t)) -(u^1,\dot{z}(0))\\
    &\hspace{0.9cm}-\int_0^t(\dot{F}(\tau),eu(\tau)-ez(\tau))\de \tau+(F(t),eu(t)-ez(t)) -(F(0),eu^0-ez(0))\\
    &\hspace{0.9cm}-\int_0^t\frac{1}{\beta}\e^{-\frac{\tau}{\beta}}(\B w^0,eu(\tau))\de \tau+(\B w^0,eu^0) -\e^{-\frac{t}{\beta}}(\B w^0,eu(t)).
\end{align*}
Therefore, the definition of total work given in~\eqref{ttwork} is coherent with the classical one~\eqref{totalwork2M}. 
\end{remark}

Now we are in position to prove the energy-dissipation inequality before mentioned. For convenience of notation we set $h(t):=\e^{-\frac{t}{\beta}}\B w^0$.

\begin{theorem}\label{energyy}
The weak solution $(u,w)\in\V\times H^1(0,T;H^d_s)$ to the coupled system \eqref{cs}--\eqref{csi}, given by Lemma \ref{lem:convM}, satisfies for every $t\in[0, T]$ the following energy-dissipation inequality 
\begin{equation}\label{eq:eninM}
\En_{u,w}(t)+\Dis_{u,w}(t)\leq \En_{u,w}(0)+\W_{tot}(t),   
\end{equation}
where $\En_{u,w}$, $\Dis_{u,w}$, and $\W_{tot}$ are defined in \eqref{energy-mecc}, \eqref{dissi}, and \eqref{ttwork}, respectively.
\end{theorem}

\begin{proof}
Fixed $t\in (0, T]$, for every $n\in\N$ there exists a unique $j\in\{1,\dots,n\}$ such that $t\in((j-1)\tau_n,j\tau_n]$. In particular, denoting by $\lceil x\rceil$ the superior integer part of the number $x$, it reads as
$$j(n)=\left\lceil \frac{t}{\tau_n}\right\rceil.$$
After setting $t_n:=j\tau_n$, we can rewrite \eqref{grwM} as follows
\begin{align}\label{qfM}
\frac{1}{2}\| \tilde{u}_n^+(t)\|^2+\frac{1}{2}(\A  eu_n^+(t),eu_n^+(t))+\frac{1}{2}(\B(eu_n^+(t)&-w_n^+(t)),eu_n^+(t)-w_n^+(t))\nonumber\\
&+\beta\int_0^{t_n} (\B\dot {w}_n(\tau),\dot {w}_n(\tau))\de \tau\leq \En_{u,w}(0)+\W_n^+(t),
\end{align}
where
\begin{align*}
\W^+_n(t):=\int_0^{t_n}[(f_n^+(&\tau),\tilde u_n^+(\tau)-\tilde z_n^+(\tau))+(F_n^+(\tau),e\tilde u_n^+(\tau)-e\tilde z_n^+(\tau))+(\dot{\tilde u}_n(\tau),{\tilde z}^+_n(\tau))]\de \tau\\   
&+\int_0^{t_n}\left[((\A +\B)  eu_n^+(\tau)-\B w_n^+(\tau),e\tilde z_n^+(\tau))-(h_n^+(\tau),e\tilde u_n^+(\tau)-e\tilde z_n^+(\tau))\right]\de \tau.
\end{align*}
Thanks to~\eqref{eq:estM} and~\eqref{dev3}, we have
\begin{align*}
\|w_n(t)-w_n^+(t)\|^2&=\|w_n^j+(t-j\tau_n)\delta w_n^j-w_n^j\|^2\leq \tau^2_n\|\delta w_n^j\|^2\leq C\tau_n\xrightarrow[n\to \infty]{}0,\\
\|u_n(t)-u_n^+(t)\|&=\|u_n^j+(t-j\tau_n)\delta u_n^j-u_n^j\|\leq \tau_n\|\delta u_n^j\|\leq C\tau_n\xrightarrow[n\to \infty]{}0,\\
\|\tilde{u}_n(t)-\tilde{u}_n^+(t)\|^2_{(V_0^D)'}&=\|\delta u_n^j+(t-j\tau_n)\delta^2 u_n^j-\delta u_n^j\|^2_{(V_0^D)'}\leq \tau^2_n\|\delta^2 u_n^j\|^2_{(V_0^D)'}\leq \tilde C\tau_n\xrightarrow[n\to \infty]{}0.
\end{align*}
The last convergences and~\eqref{prog} imply 
\begin{equation*}
u_n^+(t)\xrightharpoonup[n\to\infty]{H}u(t),\qquad w_n^+(t)\xrightharpoonup[n\to\infty]{H^d_s}w(t), \qquad \tilde{u}_n^+(t)\xrightharpoonup[n\to\infty]{(V_0^D)'}\dot{u}(t),
\end{equation*}
and since $\|u_n^+(t)\|_{V}+\|\tilde{u}_n^+(t)\|\leq C$ for every $n\in\N$, we get
\begin{align}\label{newconvM}
u_n^+(t)\xrightharpoonup[n\to \infty]{V}u(t),\qquad w_n^+(t)\xrightharpoonup[n\to\infty]{H^d_s}w(t), \qquad \tilde{u}_n^+(t)\xrightharpoonup[n\to \infty]{H}\dot{u}(t).
\end{align}
By \eqref{newconvM} and the lower semicontinuity property of the maps $v\mapsto\norm{v}^2$, $v\mapsto (\A   v, v)$, and $  v\mapsto (\mathbb{B}v, v)$, we conclude 
\begin{align}
\| \dot{u}(t)\|^2&\leq\liminf_{n\to\infty}\| \tilde{u}_n^+(t)\|^2,\label{se1M}\\
(\A e u(t),e u(t)) &\leq \liminf_{n\to\infty}(\A  e u_n^+(t),e u_n^+(t)),\\
(\mathbb{B}(e u(t)-w(t)),e u(t)-w(t)) &\leq \liminf_{n\to\infty}(\mathbb B( e u_n^+(t)-w^+_n(t)),e u_n^+(t)-w^+_n(t)).
\end{align}
Moreover, from Lemma \ref{lem:convM}, and in particular by \eqref{pf2M} we get
\begin{equation}
    \int_0^t (\mathbb B\dot w(\tau),\dot w(\tau))\de \tau\leq\liminf_{n\to \infty}\int_0^t (\B \dot{w}_n(\tau),\dot{w}_n(\tau))\de \tau\leq\liminf_{n\to \infty}\int_0^{t_n}  (\B \dot{w}_n(\tau),\dot{w}_n(\tau))\de \tau,
\end{equation}
since $t\leq t_n$ and $v\mapsto\int_0^t(\B v(\tau),v(\tau))\de \tau$ is a non negative quadratic form on $L^2(0,T;H^d_s)$. 

Now, we study the right-hand side of \eqref{qfM}. Since we have
\begin{equation*}
\chi_{[0,t_n]}f^+_n\xrightarrow[n\to\infty]{L^2(0, T;H)}\chi_{[0,t]}f\quad\text{and}\quad \tilde{u}^+_n-\tilde{z}^+_n\xrightharpoonup[n\to\infty]{L^2(0, T;H)}\dot{u}-\dot z,
\end{equation*}
we deduce that
\begin{align}
 \int_0^{t_n}(f^+_n(\tau),\tilde{u}^+_n(\tau)-\tilde{z}^+_n(\tau))\de \tau \xrightarrow[n\to\infty]{} \int_0^t(f(\tau),\dot{u}(\tau)-\dot z(\tau))\de \tau .
\end{align}
In a similar way, since the following convergences hold
\begin{equation*}
\chi_{[0,t_n]}e\tilde{z}^+_n\xrightarrow[n\to\infty]{L^2(0,T;H^d_s)}\chi_{[0,t]}e\dot{z},\qquad h_n^+\xrightarrow[n\to\infty]{L^2(0,T;H^d_s)}h, \qquad (\A +\B) e u^+_n-  \B w^+_n\xrightharpoonup[n\to\infty]{L^2(0,T;H^d_s)} (\A +\B) e u- \B w,
\end{equation*}
we obtain
\begin{align}
\int_0^{t_n}(h_n^+(\tau),e\tilde z_n^+(\tau))\de \tau &\xrightarrow[n\to\infty]{}  \int_0^t (h(\tau),e\dot z(\tau))\de\tau \\
\int_0^{t_n}((\A +\B)  e u^+_n(\tau)-\B w_n^+(\tau),e\tilde{ z}^+_n(\tau))\de \tau  &\xrightarrow[n\to\infty]{} \int_0^{t}((\A +\B)  e u(\tau)- \B w(\tau),e\dot{ z}(\tau))\de \tau.
\end{align} 
By means of the discrete integration by parts formulas \eqref{dis-part}--\eqref{dis-part2.1} we can write
\begin{align}
     \int_0^{t_n}(\dot{\tilde u}_n(\tau),\tilde z_n^+(\tau))\de \tau &= (\tilde u_n^+(t),\tilde z_n^+(t)) -(u^1,\dot z(0))- \int_0^{t_n}(\tilde u_n^-(\tau),\dot{\tilde z}_n(\tau))\de \tau ,\label{n1}\\
    \int_0^{t_n}(h_n^+(\tau),e\tilde u_n^+(\tau))\de \tau &= (eu_n^+(t),h_n^+(t)) -(eu^0,h(0))- \int_0^{t_n}(\tilde h_n^+(\tau),e u_n^-(\tau))\de \tau ,\\
    \int_0^{t_n}(F_n^+(\tau),e\tilde u_n^+(\tau)-e\tilde z_n^+(\tau))\de \tau& = (F_n^+(t),eu_n^+(t)- ez_n^+(t)) -( F(0),eu^0- ez(0))\nonumber\\
    &\hspace{2.7cm}- \int_0^{t_n}(\tilde F_n^+(\tau),eu_n^-(\tau)- ez_n^-(\tau))\de \tau\label{n2}.
\end{align}
Notice that the following convergences hold
\begin{align*}
\|\tilde{z}_n^+(t)-\dot{z}(t)\|&=\Big\|\frac{z(j\tau_n)-z((j-1)\tau_n)}{\tau_n}-\dot{z}(t)\Big\|\leq \int_{(j-1)\tau_n}^{j\tau_n}\|\dot{z}(\tau)-\dot{z}(t)\|\de \tau\xrightarrow[n \to \infty]{}0,\\    
\|h_n^+(t)-h(t)\|&=\|\B w^0\||\e^{-\frac{j\tau_n}{\beta}}-\e^{-\frac{t}{\beta}}|\leq \frac{1}{\beta^2}\|\B w^0\||t-j\tau_n|\leq \frac{1}{\beta^2}\|\B w^0\|\tau_n \xrightarrow[n \to \infty]{}0,\\
\|z_n^+(t)-z(t)\|_{V}&=\|z(j\tau_n)-z(t)\|_{V}\leq (j\tau_n-t)^{\frac{1}{2}}\|\dot z\|_{L^2(0,T;V)}\leq \tau_n^{\frac{1}{2}}\|\dot z\|_{L^2(0,T;V)}\xrightarrow[n \to \infty]{}0,\\
\|F_n^+(t)-F(t)\|&=\|F(j\tau_n)-F(t)\|\leq (j\tau_n-t)^{\frac{1}{2}}\|\dot F\|_{L^2(0,T;H^d_s)}\leq \tau_n^{\frac{1}{2}}\|\dot F\|_{L^2(0,T;H^d_s)}\xrightarrow[n \to \infty]{}0,
\end{align*}
\begin{align*}
     \chi_{[0,t_n]}\dot{\tilde z}_n&\xrightarrow[n\to\infty]{L^2(0,T;H)}\chi_{[0,t]}\ddot{z},\quad \chi_{[0,t_n]}\tilde h_n^+\xrightarrow[n\to\infty]{L^2(0,T;H^d_s)}\chi_{[0,t]}\dot{h},\nonumber \\
z_n^{-}&\xrightarrow[n\to\infty]{L^2(0,T;V)}z,\quad\hspace{0.75cm}\chi_{[0,t_n]}{\tilde F}^+_n\xrightarrow[n\to\infty]{L^2(0,T;H^d_s)}\chi_{[0,t]}\dot F.
\end{align*}
By means of these convergences, \eqref{newconvM}, and Lemma \ref{lem:convM}, we can argue as before to deduce from \eqref{n1}--\eqref{n2} 
\begin{align}
  \int_0^{t_n}(\dot{\tilde u}_n(\tau),\tilde z_n^+(\tau))\de \tau &\xrightarrow[n\to \infty]{} (\dot u(t),\dot z(t)) -(u^1,\dot z(0))- \int_0^{t}(\dot u(\tau),{\ddot z}(\tau))\de \tau ,\\
 \int_0^{t_n}(h_n^+(\tau),e\tilde u_n^+(\tau))\de \tau &\xrightarrow[n\to \infty]{} (h(t),e u(t)) -(h(0),eu^0)- \int_0^{t}(\dot h(\tau),eu(\tau))\de \tau ,\\
  \int_0^{t_n}(F_n^+(\tau),e\tilde u_n^+(\tau)-e\tilde z_n^+(\tau))\de \tau& \xrightarrow[n\to \infty]{}(F(t),eu(t)- ez(t)) -(F(0),eu^0- ez(0))\nonumber\\
  &\hspace{2.9cm}- \int_0^{t}(\dot F(\tau),eu(\tau)- ez(\tau))\de \tau.\label{se-ultima}
\end{align}

By combining \eqref{qfM} and \eqref{se1M}--\eqref{se-ultima} we obtain the energy-dissipation inequality \eqref{eq:eninM} for $t\in(0,T]$. Finally, for $t=0$ the inequality trivially holds since $u(0)=u^0$ and $\dot u(0)=u^1$. 
\end{proof}

\begin{remark}
Thanks to the last theorem and to the equivalence between the viscoelastic dynamic system \eqref{classic_model_inf3}--\eqref{sys_inf3} and the coupled system \eqref{cs}--\eqref{csi}, we can derive an energy-dissipation inequality for a weak solution to our viscoelastic dynamic system \eqref{classic_model_inf3}--\eqref{sys_inf3}. As can be seen from \eqref{wweak} and the proof of Theorem \ref{equivalence} it is not restrictive to assume $w^0=0$.

Let $(u,w)$ be the weak solution to the coupled system \eqref{cs}--\eqref{csi} provided by Lemma \ref{lem:convM}. Then, it satisfies the energy-dissipation inequality \eqref{eq:eninM}. Moreover, from Theorem \ref{equivalence} the function $u$ is a solution to the viscoelastic dynamic system \eqref{classic_model_inf3}--\eqref{sys_inf3} in the sense of Definition \ref{newdef2}. Therefore, by substituting \eqref{sol-ode} in \eqref{eq:eninM} we get for the conservative part
\begin{align}\label{1}
    \En_{u,w}(t)&=\frac{1}{2}\|\dot u(t)\|^2 + \frac{1}{2}(\A   eu(t),eu(t)) + \frac{1}{2}(\B(eu(t)- w(t)),eu(t)-w(t))\nonumber\\
    &=\frac{1}{2}\|\dot u(t)\|^2 + \frac{1}{2}((\A +\B)  eu(t),eu(t))-\int_0^t \frac{1}{\beta} \e^{-\frac{t-\tau}{\beta}}(\B eu(\tau),eu(t))\de \tau\nonumber\\
    &\hspace{1.9cm}+\frac{1}{2\beta^2}\int_0^t\int_0^t\e^{-\frac{2t-r-\tau}{\beta}}(\B eu(r),eu(\tau))\de r \de \tau
\end{align}
and for the dissipation
\begin{align}\label{105}
    \Dis_{u,w}(t)&=\int_0^t(\B\dot w(\tau),eu(\tau)- w(\tau))\de \tau=\int_0^t(\B\dot w(\tau),eu(\tau))\de \tau-\int_0^t(\B\dot w(\tau),w(\tau))\de \tau\nonumber\\
   &=\frac{1}{\beta}\int_0^t\Big(\B eu(\tau)-\int_0^{\tau}\frac{1}{\beta}\e^{-\frac{\tau-r}{\beta}}\B e u(r)\de r,eu(\tau)\Big)\de \tau-\frac{1}{2}(\B w(t),w(t))\nonumber\\
   &=\frac{1}{\beta}\int_0^t(\B eu(\tau),eu(\tau))\de \tau-\frac{1}{\beta^2}\int_0^t\int_0^{\tau}\e^{-\frac{\tau-r}{\beta}}(\B e u(r),eu(\tau))\de r\de\tau\nonumber\\
   &\hspace{4.3cm}-\frac{1}{2\beta^2}\int_0^t\int_0^t\e^{-\frac{2t-r-\tau}{\beta}}(\B eu(r),eu(\tau))\de r\de \tau.
\end{align}
By substituting the same information in the total work, we obtain
\begin{align}\label{2}
  \W_{tot}(t)&=\int_0^t \left[(f(\tau),\dot u(\tau)-\dot z(\tau))+((\A +\B)  eu(\tau),e\dot z(\tau))-\int_0^{\tau}\frac{1}{\beta}\e^{-\frac{\tau-r}{\beta}}(\B eu(r),e\dot z(\tau))\de r\right]\de \tau\nonumber\\
  &\hspace{0.2cm}-\int_0^t(\dot F(\tau),eu(\tau)-ez(\tau))\de \tau+( F(t),eu(t)-ez(t))-( F(0),eu^0-ez(0))\nonumber\\
  &\hspace{0.2cm}-\int_0^t(\dot u(\tau),\ddot z(\tau))\de \tau+(\dot{u}(t),\dot{z}(t)) -(u^1,\dot{z}(0)).
\end{align}
After defining the elastic energy as
\begin{align*}
\En(t):=\frac{1}{2}\|\dot u(t)\|^2&+\frac{1}{2}((\A +\B)  eu(t),eu(t))\nonumber\\
    &-\int_0^t \frac{1}{\beta} \e^{-\frac{t-\tau}{\beta}}(\B eu(\tau),eu(t))\de \tau+\frac{1}{2\beta^2}\int_0^t\int_0^t\e^{-\frac{2t-r-\tau}{\beta}}(\B eu(r),eu(\tau))\de r\de \tau,
\end{align*}
and the dissipative term
\begin{align*}
    \Dis(t):=\frac{1}{\beta}\int_0^t(\B eu(\tau),eu(\tau))\de \tau&-\frac{1}{\beta^2}\int_0^t\int_0^{\tau}\e^{-\frac{\tau-r}{\beta}}(\B eu(r),eu(\tau))\de r\de \tau \nonumber\\
    &-\frac{1}{2\beta^2}\int_0^t\int_0^t\e^{-\frac{2t-r-\tau}{\beta}}(\B eu(r),eu(\tau))\de r\de \tau,
\end{align*}
taking into account \eqref{1}, \eqref{105}, and \eqref{2} we can rephrase the energy-dissipation inequality \eqref{eq:eninM} as
\begin{equation*}
    \En(t)+\Dis(t)\leq \En(0)+\W_{tot}(t),
\end{equation*}
where the total work $\W_{tot}$ now depends just on the function $u$.
\end{remark}

Finally, in view of Theorem \ref{energyy} we are ready to show that our weak solution satisfies the initial conditions in a stronger sense than the one stated in \eqref{ini*}, that is the content of the following lemma.

\begin{lemma}\label{deidat}
The weak solution $(u,w)\in\V\times H^1(0,T;H^d_s)$ to the coupled system \eqref{cs}--\eqref{csi}, provided by Lemma \ref{lem:convM}, satisfies the initial conditions in the following sense:
\begin{equation}\label{eq:inconM}
    \lim_{t\to 0^+}u(t)= u^0\text{ in $V$},\quad\lim_{t\to 0^+}\dot u(t)=u^1\text{ in $H$},\quad \lim_{t\to 0^+} w(t)=w^0\text{ in $H^d_s$}.
\end{equation}
\end{lemma}
\begin{proof}
Since $u\in C_w^0([0, T];V)$, $\dot u\in C_w^0([0, T];H)$, $w\in C^0([0,T];H^d_s)$, from the lower semicontinuity of the real valued functions 
$$t\mapsto \norm{\dot u(t)}^2,\quad t\mapsto (\A  eu(t),eu(t)),\quad t\mapsto (\B (eu(t)-w(t)),eu(t)-w(t)),$$
we can let $t\rightarrow0^+$ into the energy-dissipation inequality \eqref{eq:eninM} to deduce that
\begin{align}\label{star*}
    \En_{u,w}(0)&=\frac{1}{2}\norm{u^1}^2+\frac{1}{2}(\A   eu^0,eu^0)+\frac{1}{2}(\B(eu^0- w^0),eu^0-w^0)\nonumber\\
    &\le\frac{1}{2}\Big[\liminf_{t\to 0^+}\norm{\dot u(t)}^2+\liminf_{t\to 0^+}(\A   eu(t), eu(t))+\liminf_{t\to 0^+}(\B (eu(t)-w(t)),eu(t)-w(t))\Big]\nonumber\\
    &\leq \liminf_{t\to 0^+}\Big[\frac{1}{2}\norm{\dot u(t)}^2+\frac{1}{2}(\A   eu(t), eu(t))+\frac{1}{2}(\B(eu(t)- w(t)),eu(t)-w(t))\Big]\nonumber\\
    &=\liminf_{t\to 0^+}\En_{u,w}(t)\leq \limsup_{t\to 0^+}\En_{u,w}(t)\leq \En_{u,w}(0).
\end{align}
Notice that the last inequality in \eqref{star*} holds because the right-hand side of \eqref{eq:eninM} is continuous in $t$, and $u(0)=u^0$, $\dot u(0)=u^1$, and $w(0)=w^0$. Therefore, there exists $\lim_{t\to 0^+}\En_{u,w}(t)=\En_{u,w}(0)$. Moreover, we have
\begin{align*}
    \En_{u,w}(0)&\le\frac{1}{2}\liminf_{t\to 0^+}\norm{\dot u(t)}^2+\frac{1}{2}\liminf_{t\to 0^+}\Big[(\A   eu(t), eu(t))+(\B (eu(t)-w(t)),eu(t)-w(t))\Big]\nonumber\\
    &\leq \frac{1}{2}\limsup_{t\to 0^+}\norm{\dot u(t)}^2+\frac{1}{2}\liminf_{t\to 0^+}\Big[(\A   eu(t), eu(t))+(\B (eu(t)-w(t)),eu(t)-w(t))\Big]\nonumber\\
    &\leq\limsup_{t\to 0^+}\Big[\frac{1}{2}\norm{\dot u(t)}^2+\frac{1}{2}(\A   eu(t), eu(t))+\frac{1}{2}(\B(eu(t)- w(t)),eu(t)-w(t))\Big]=\En_{u,w}(0),
\end{align*}
which gives
\begin{equation*}
    \lim_{t\to 0^+}\norm{\dot u(t)}^2=\norm{u^1}^2.
\end{equation*}
In a similar way, we can also show that
\begin{equation*}
    \lim_{t\to 0^+}(\A   eu(t),eu(t))=(\A   eu^0,eu^0).
\end{equation*}

Finally, since we have 
$$\dot u(t)\xrightharpoonup[t\to 0^+]{H}u^1,\quad  eu(t)\xrightharpoonup[t\to 0^+]{H^d_s}eu^0$$ 
and $u\in C^0([0, T];H)$, we deduce \eqref{eq:inconM}. In particular the functions $u\colon [0, T]\to V$ and $\dot u\colon [0, T]\to H$ are continuous at $t=0$. 
\end{proof}

We can finally prove the main theorem of Section \ref{sub2}.

\begin{proof}[Proof of Theorem~\ref{main2}]
It is enough to combine Proposition~\ref{prop-main} and Lemma~\ref{deidat}.
\end{proof}

\begin{remark}
We have proved Theorem~\ref{main2} for the $d$-dimensional linear viscoelastic case, namely when the displacement $u$ is a vector-valued function. The same result is true with identical proof in the \emph{antiplane} case, that is when the displacement $u$ is a scalar function and satisfies~\eqref{viscoelasticM}.
\end{remark}

{\begin{acknowledgements}
The author wishes to thank Professors Gianni Dal Maso and Rodica Toader for having proposed the problem and for many helpful discussions on the topic. The author is a member of the {\em Gruppo Nazionale per l'Analisi Ma\-te\-ma\-ti\-ca, la Probabilit\`a e le loro Applicazioni} (GNAMPA) of the {\em Istituto Nazionale di Alta Matematica} (INdAM).
\end{acknowledgements}}

\end{document}